\documentclass[11pt]{amsart}
\usepackage{amsmath,amssymb,amsthm, amscd, mathtools}
\usepackage[english,french]{babel}

\newcommand\Fp{\mathbb{F}_p}

\newcommand\Z{\mathbb{Z}}

\newcommand\Q{\mathbb{Q}}

\def\p{{\mathfrak{p}}}

\newenvironment{Gal}{{\rm Gal\,}}{}

\newenvironment{Disc}{{\rm Disc}}{}

\newtheorem{theorem}{Theorem}[section]
\newtheorem{definition}[theorem]{Definition}
\newtheorem{lemma}[theorem]{Lemma}
\newtheorem{proposition}[theorem]{Proposition}
\newtheorem{proposition-definition}[theorem]{Proposition-Definition}
\newtheorem{corollary}[theorem]{Corollary}
\newtheorem{conjecture}[theorem]{Conjecture}

\newtheorem{question}[theorem]{Question} 

\theoremstyle{definition}

\theoremstyle{remark}
\newtheorem*{remark}{Remark}

\title[Newly reducible polynomial iterates]{Newly reducible polynomial iterates}

\subjclass[2010]{Primary: 37P15. Secondary: 11R09, 37P05, 37P25}

\author{Peter Illig, Rafe Jones, Eli Orvis, Yukihiko Segawa, Nick Spinale}

\numberwithin{equation}{section}

\begin{document}

%\begin{resume}
%Pour un corps $K$ et entier $n > 1$, nous disons qu'un polynome $f \in K[x]$ a un 
%\end{resume}

\begin{abstract} 
Given a field $K$ and $n > 1$, we say that a polynomial $f \in K[x]$ has newly reducible $n$th iterate over $K$ if $f^{n-1}$ is irreducible over $K$, but $f^n$ is not (here $f^i$ denotes the $i$th iterate of $f$). We pose the problem of characterizing, for given $d,n > 1$, fields $K$ such that there exists $f \in K[x]$ of degree $d$ with newly reducible $n$th iterate, and the similar problem for fields admitting infinitely many such $f$. We give results in the cases $(d,n) \in \{(2,2), (2,3), (3,2), (4,2)\}$ as well as for $(d,2)$ when $d \equiv 2 \bmod{4}$. In particular, we show that for all these $(d,n)$ pairs, there are infinitely many monic $f \in \Z[x]$ of degree $d$ with newly reducible $n$th iterate over $\Q$. Curiously, the minimal polynomial $x^2 - x - 1$ of the golden ratio is one example of $f \in \Z[x]$ with newly reducible third iterate; very few other examples have small coefficients. Our investigations prompt a number of conjectures and open questions.

%that is irreducible over $K$, one may ask whether the $n$th iterate $f^n$ of $f$ must remain irreducible over $K$ for all $n > 1$. A significant literature has developed giving sufficient conditions on $f$ to ensure this happens. Here we study the complementary problem of finding irreducible $f \in K[x]$ for which there is some $n > 1$ such that $f, f^2, \ldots, f^{n-1}$ are all irreducible over $K$, but $f^n$ is not. We say that such $f$ has a newly reducible $n$th iterate. 

%It is known that there are infinitely many quadratic $f \in \Q[x]$ with newly reducible 2nd iterate, and indeed the same statement remains true if we require $f$ to be monic with integer coefficients. We prove here that similar statements hold for quadratic polynomials and newly reducible 3rd iterate, and also for cubic and quartic polynomials with newly reducible 2nd iterate. We further show that there are infinitely many degrees $d$ such that there exist infinitely many monic polynomials of degree $d$ with integer coefficients and newly reducible 2nd iterate. 
\end{abstract}

\maketitle

\section{Introduction}
Although the mathematical study of the golden ratio dates to antiquity, one of its unusual properties appears to have passed the millennia unnoticed: its minimal polynomial $f(x) = x^2 - x - 1$ is irreducible over $\Q$, as is the second iterate $f(f(x))$ of this polynomial, but 
\begin{equation} \label{goldenfact}
f(f(f(x))) %= x^8 - 4x^7 + 14x^5 - 7x^4 - 14x^3 + 7x^2 + 3x - 1 
= (x^4 - 3x^3 + 4x - 1)(x^4 - x^3 - 3x^2 + x + 1),
\end{equation}
where the two quartics are irreducible. The factorization in \eqref{goldenfact} together with the irreducibility over $\Q$ of $f(f(x))$ is a rare phenomenon: among polynomials $x^2 + ax + b$ with $a, b$ integers satisfying $|a| \leq 100,000$ and $|b| \leq 1,000,000,000$, only 8 others have this property. One might wonder whether there are infinitely many such polynomials. Indeed there are, as we show 
%and we give a description of the collection of all of them 
in Theorem \ref{summary} below. 
%It is far from clear whether the set $\{x^2 + ax + b : a, b \in \Z\}$ of polynomials that share this property is infinite. 

More generally, let $K$ be a field and $f \in K[x]$ a polynomial of degree $d \geq 2$. Write $f^n(x)$ for the $n$th iterate of $f$. We say that $f$ has a \textit{newly reducible $n$th iterate} (over $K$) for some $n \geq 2$ if $f^{n-1}$ is irreducible over $K$ but $f^n$ is not irreducible over $K$. We note that the irreducibility of $f^{n-1}$ implies that each of $f, f^2, \ldots, f^{n-2}$ is also irreducible over $K$. Several recent papers have given conditions ensuring that all iterates of a polynomial remain irreducible (see e.g. \cite{ayad}, \cite{ayadcor}, \cite{danielson}, \cite{goksel}, \cite{quaddiv}, \cite{itconst}). Many fewer have studied newly reducible iterates; a few examples are \cite{HCsummer}, \cite{fein}, and \cite{preszler}. In this paper our main objects of study are the following.

\begin{definition}
Let $d,n$ be integers that are at least two. Define 
$\mathcal{N}_{d,n}$ (resp. $\mathcal{N}_{d,n}^{\infty})$ to be the class\footnote{Because the collection of all fields is not a set, neither are $\mathcal{N}_{d,n}$ and $\mathcal{N}_{d,n}^\infty$. We use them in this article not as objects, but as notational devices. In referring to them we use the usual notation and language of set theory.} of all fields $K$ such that there exists at least one (resp. infinitely many) $f \in K[x]$ of degree $d$ with newly reducible $n$th iterate over $K$. 
\end{definition}

When $K$ is a perfect field, the question of whether $K \in \mathcal{N}_{d,n}$ can be phrased as an inverse Galois-type problem: does there exist $f \in K[x]$ such that the absolute Galois group of $K$ acts transitively on the roots of $f^{n-1}$ but not on the roots of $f^n$? Fein and Schacher \cite{fein} appear to be the first to have studied newly reducible iterates, and they used results of Odoni to obtain the fundamental result in this area: for all $d,n \geq 2$, $\mathcal{N}_{d,n}$ is non-empty (\cite[Corollary 1.3]{fein}). Their method relies on knowledge of the Galois groups of iterates of generic polynomials, and then an appeal to the Hilbert irreducibility theorem; the resulting field $K$ depends on the choice of specialization. Our aim is to address at least some cases of the following questions, each of which is unresolved by the methods of \cite{fein}.

%Throughout this article, all irreducibility statements are over $\Q$ unless explicitly stated otherwise. 

\begin{question} \label{mainquest}
Let $d,n$ be integers that are at least $2$. 
\begin{enumerate}
\item Precisely which fields belong to $\mathcal{N}_{d,n}$ and $\mathcal{N}_{d,n}^\infty$?
\item For which $d,n$ is $\Q \in \mathcal{N}_{d,n}$? For which $d,n$ is $\Q \in \mathcal{N}_{d,n}^\infty$? 
%\item Is $\mathcal{N}_{d,n}^\infty$ non-empty? Precisely which fields belong to it?
\item If $K$ is a number field with ring of integers $\mathcal{O}_K$ and $K \in \mathcal{N}_{d,n}^\infty$, do there exist infinitely many monic $f \in \mathcal{O}_K[x]$ of degree $d$ with newly reducible $n$th iterate over $K$?
\end{enumerate}
\end{question}

%Fein and Schacher \cite{fein} were probably the first to study Question \ref{mainquest}, though their primary results concern the question of whether for fixed $d, n \geq 2$ there exists a field $K$ and a polynomial $f \in K[x]$ of degree $d$ with newly reducible $n$th iterate. They give an affirmative answer (\cite[Corollary 1.3]{fein}), but their method involves fixing $f$ and constructing $K$ in a way that depends delicately on $f$. Thus they only partially address Question \ref{mainquest}. 

%In this paper we consider Question \ref{mainquest} in the case $K = \Q$. We thus make the stipulation that 
%$$
%\textit{throughout this paper, all irreducibility statements are over $\Q$ unless explicitly stated otherwise.}
%$$
The case $K = \mathbb{Q}$ has received significant attention recently. The recent paper \cite{goksel} of Goksel describes an infinite family of monic quadratic $f \in \Z[x]$ with newly reducible third iterate \cite[Lemma 3.9]{goksel}, thereby showing that $\Q \in \mathcal{N}_{2,3}$. % and hence $\mathcal{N}_{2,3}^\infty \neq \emptyset$. 
The paper \cite{preszler} gives an infinite family of non-monic cubic $f \in \Z[x]$ with newly reducible second iterate, thus proving $\Q \in \mathcal{N}_{3,2}$, 
%and $\mathcal{N}_{3,2}^\infty \neq \emptyset$, 
and conjectures that no such family exists among monic cubic $f \in \Z[x]$ \cite[Conjecture 4.1]{preszler}. 

We summarize our main results. Throughout the article, we sometimes use $\text{char}(K)$ to denote the characteristic of a field $K$. 

\begin{theorem} \label{summary}
Let $K$ be an infinite field with $\text{char}(K) \neq 2$, and for $n \geq 2$ set $K^n = \{k^n : k \in K\}$. 
\begin{enumerate}
\item We have $K \in \mathcal{N}_{2,2}$ if and only if $K \neq K^{2}$, and $K \in \mathcal{N}_{2,2}^\infty$ if and only if $K \setminus K^{2}$ is infinite (Proposition \ref{n22}).
\item We have $K \in \mathcal{N}_{2,3}^\infty$ if 
\begin{itemize}
\item $-1 \not\in K^2$ and $K$ has a discrete valuation $v$ with $v(5)$ odd (Proposition \ref{m-1}); or
\item $K$ is a totally real number field and there is a prime of residue degree 1 lying over $(3)$ (Corollary \ref{newfamily}).
\end{itemize}
\item We have $K \in \mathcal{N}_{3,2}^\infty$ if $\text{char}(K) \neq 3$ and $2 \not\in K^3$ (Theorem \ref{cubic2}).
\item We have $K \in \mathcal{N}_{4,2}^\infty$ if $\text{char}(K) \neq 3$, $3 \not\in K^2$, and $-3 \not\in K^4$ (Corollary \ref{deg-4-infinitely-many-corollary}).
\item We have $K \in \mathcal{N}_{d,2}^\infty$ for all $d \equiv 2 \pmod{4}$ if $-1 \not\in K^2$ and $K$ has a non-trivial discrete valuation (Corollary \ref{genbigd}).
\end{enumerate}
\end{theorem}
In each case, we exhibit an explicit infinite family of polynomials with the desired properties, and if $K$ is a number field the polynomials may be taken to be monic with coefficients in $\mathcal{O}_K$ without loss of generality. For instance, a consequence of Theorem \ref{cubic2} is that if
\begin{equation} \label{cubicfamily}
f(x) = \left(x + 93312 t^9+36 t^3\right)^3-93312 t^9
\end{equation}
for $t \in \Z \setminus \{0\}$, then $f$ has a newly reducible second iterate over $\Q$. This disproves Conjecture 4.1 of \cite{preszler}, which is based on a search of all polynomials of the form $x^3 + ax^2 + bx + c,$ where $a,b,c$ are integers with absolute value at most 500; it is noteworthy that none of the polynomials in \eqref{cubicfamily} are of this form.  

When $K$ is a finite field of characteristic not equal to $2$, all the statements of Theorem \ref{summary} hold with $\mathcal{N}_{d,n}^\infty$ replaced by $\mathcal{N}_{d,n}$, although (5) and the first statement of (2) become vacuous.  We have $\mathbb{F}_2 \not\in \mathcal{N}_{2,2}$, since $f(x) = x^2 + x + 1$ is the only irreducible quadratic over $\mathbb{F}_2$, and $f^2(x)$ is irreducible over $\mathbb{F}_2$. We prove in Proposition \ref{fintwo} that $\mathbb{F}_{2^n} \in \mathcal{N}_{2,2}$ for all $n > 1$ and $\mathbb{F}_{2^n} \in \mathcal{N}_{2,3}$ for all $n \geq 1$. Surprisingly, $\mathbb{F}_{2^n} \not\in \mathcal{N}_{2,k}$ for any $k \geq 4$ and any $n \geq 1$, due to a result of Ahmadi et al \cite[Theorem 10]{shparostafe} (see also \cite{ahmadi}) that every $f \in \mathbb{F}_{2^n}[x]$ has $f^3(x)$ reducible. 

When $K$ is a general field of characteristic 2, we know of no results addressing whether $K \in \mathcal{N}_{d,n}$ for any $d,n$, though we give a condition that applies to the case $d = n = 2$ in Proposition \ref{secondchar2}. This leads us to pose the following question.

\begin{question} \label{charpquest}
Let $p$ be a prime number and $d \geq p$. Which fields of characteristic $p$ belong to $\mathcal{N}_{d,n}$ for various $n \geq 2$?
\end{question}

The results of Theorem \ref{summary} prompt the following two conjectures.
\begin{conjecture}
For each $d, n \geq 2$, $\mathcal{N}_{d,n}^\infty$ is non-empty. 
\end{conjecture}

\begin{conjecture}
Let $K$ be a finite extension of $\Q$ or $\Fp(t)$. Then $K \in \mathcal{N}_{2,3}^\infty$ and $K \in \mathcal{N}_{d,2}^\infty$ for every $d \geq 2$. 
\end{conjecture}

The proofs of the statements in Theorem \ref{summary} proceed by enumerating all $f$ of a specified form such that $f^n(x)$ is reducible, and then giving conditions under which infinitely many $f$ remain after discarding those with $f^m$ reducible for $m<n$. In the case where $f$ is quadratic, the specified form is $f(x) = (x-\gamma)^2 + \gamma + m$, so that $\gamma$ is the critical point of $f$, and this includes all $f$. We thus obtain the following result. 

%Our first main result gives, for fields $K$ of characteristic not equal to $2$, a description of all quadratic $f \in K[x]$ with newly reducible third iterate. In this result, and throughout the paper, we write a quadratic polynomial in the form 
\begin{theorem} \label{characterization}
Let $K$ be a field of characteristic not equal to $2$, and let $f\left(x\right)=\left(x-\gamma\right)^2+\gamma+m\in K[x]$. Then $f$ has a newly reducible third iterate over $K$ if and only if the following both hold:
%$f^3\left(x+\gamma\right)$ is symmetrically reducible if and only if $\gamma=-m$ or there exist $r,s\in\Q$ such that
\begin{enumerate}
\item There exist $r,s\in K$ such that
\begin{align*}
\gamma & = \frac{1}{256r^2}(-2 r^5 s^2+9 r^4 s^4-4 r^4 s^2-16 r^3 s^6+32 r^3 s^4+16 r^3 s^2+32 r^3+14 r^2 s^8 \\ &\phantom{=\frac{1}{256r^2}(}-64 r^2 s^6-8 r^2 s^4+96 r^2 s^2+128 r^2-6 r s^{10}+48 r s^8-64 r s^6 \nonumber \\ &\phantom{=\frac{1}{256r^2}(} -160 r s^4+96 r s^2+128 r+s^{12}-12 s^{10}+40 s^8-112 s^4-64 s^2), \nonumber \\
m & = \frac{-r^2-2 r s^2-4 r+s^4-4 s^2-4}{8 r}.
\end{align*}
\item None of $\sqrt{-m-\gamma}$, $\sqrt{-2m + 2\sqrt{m^2+m+\gamma}}$, $\sqrt{-2m - 2\sqrt{m^2+m+\gamma}}$ is in $K$.
\end{enumerate}
\end{theorem}
We note that taking $K = \Q$ and $r = s = 1$ in Theorem \ref{characterization} gives $\gamma = 1/2$ and $m = -7/4$. This yields $f(x) = (x-1/2)^2 + 1/2 - 7/4 = x^2 - x - 1$, and thus we recover the minimal polynomial of the golden ratio, mentioned in the first paragraph. Observe that for this polynomial, $-m - \gamma = 5/4$ and $m^2 + m + \gamma = 29/16$, ensuring that condition (2) in Theorem \ref{characterization} holds. 

%[Comment on how this compares to the Involve result?]

In general, a convenient way to verify the conditions in part (2) of Theorem \ref{characterization} is to show that neither $-m - \gamma$ nor $m^2 + m + \gamma$ is a square in $K$. With $\gamma, m$ as in part (1) of Theorem \ref{characterization}, up to squares in $K(r,s)$ we have $-m-\gamma = 2r - s^2 + 4$ and
$$
m^2 +m + \gamma = -2s^2r^3 + (5s^4 + 4s^2 + 4)r^2 + (-4s^6 + 12s^4 + 32s^2 + 16)r + s^8 - 8s^6 + 8s^4 + 32s^2 + 16
$$ 

In \cite[Lemma 3.9]{goksel}, Goksel studies the case $m = -1$ and gives algebraic characterizations of all $\gamma$ making $f^2(x)$ and $f^3(x)$ reducible. We recover his results as an outcome of our proof of Theorem \ref{characterization} (see the discussion following the proof of Theorem \ref{symmchar}). We use Theorem \ref{characterization} to give another family with non-constant $m$ in Theorem \ref{newfamily}. %, and show that Goksel's family is the only one with constant $m$. [This last statement is false. Redo this.]

We close this introduction with three additional questions, and an outline of the paper. 

\begin{question} \label{fourquest}
Is $\Q \in \mathcal{N}_{2,4}$? 
\end{question}

%It is interesting to ask whether there exists quadratic $f \in \Q[x]$ with newly reducible fourth iterate. 
We have not been able to find any quadratic $f \in \Q[x]$ with newly reducible fourth iterate; we briefly discuss some work on this question in Section \ref{fourth}. We remark that  \cite[Theorem 1.2]{goksel} shows that for $m \in \{0,-1,-2\}$ (i.e. when $f$ is post-critically finite), there is no $\gamma \in \Z$ that works. 

%We also study monic polynomials in higher degree with newly reducible second iterate. 
%\begin{theorem} \label{maincubquad}
%The polynomials in the following families have a newly reducible second iterate over $\Q$ for any $t \in \Z$, $t \neq 0$:
%\begin{enumerate}
%\item $f(x) = \left(x + 93312 t^9+36 t^3\right)^3-93312 t^9$
%\item $f(x) = \left(x - 4588382453760 t^8 - 744 t^2\right)^4 + 4588382453760 t^8$
%\end{enumerate} 
%\end{theorem}

%We remark that the families in Theorem \ref{maincubquad} are of the form $(x - \gamma)^d + \gamma + m$, and hence are unicritical with unique critical point $\gamma$. 

\begin{question} \label{cubicquest}
Does there exist $f \in \Q[x]$ such that $\deg f = 3$ and $f^2$ is newly reducible with three distinct factors, each of degree $3$?
\end{question}

In the course of our study of cubic polynomials in Section \ref{cubics} we find infinite families with newly reducible second iterate, but all such families with $f \in \Q[x]$ have $f^2(x)$ that factors as a product of an irreducible sextic and cubic. It would be interesting to find an $f \in \Q[x]$ with second iterate factoring as in Question \ref{cubicquest}. More generally, it would be interesting to study polynomials with newly reducible iterates that have more than two irreducible factors. 

To motivate our last question, we remark that the fields in $\mathcal{N}_{d,n}$ constructed in \cite{fein} are number fields whose degrees grow very rapidly with $n$ and $d$. 

\begin{question}
Fix $d \geq 2$, and let $m_n$ be the minimal degree of a number field $K$ over which there is $f \in K[x]$ of degree $d$ with newly reducible $n$th iterate. Is $(m_n)_{n \geq 2}$ unbounded? Does $\lim_{n \to \infty} m_n = \infty$?
\end{question}

In Section \ref{symmetric}, we prove under mild hypotheses that if $f$ has a newly reducible $n$th iterate, then the factorization of $f^n$ must have a certain form. This generalizes some of the results from \cite{settled}. These results are particularly useful in the case where $f$ is quadratic, which we study in Section \ref{surface}, proving Theorem \ref{characterization} and parts (1) and (2) of Theorem \ref{summary}, as well as some results for fields of characteristic 2, such as Proposition \ref{fintwo}. Section \ref{fourth} contains brief remarks on Question \ref{fourquest}. In Section \ref{cubics}, we study cubics with newly reducible second iterate, and prove part (3) of Theorem \ref{summary}. Sections \ref{quartics} and \ref{largerdegrees} study polynomials of higher degree with newly reducible second iterate, and contain the proofs of parts (4) and (5) of Theorem \ref{summary}.

\medskip

\noindent \textbf{Acknowledgements.} We thank the anonymous referee for helpful comments. We are also grateful to Carleton College's Towsley Endowment for the Sciences, which partially supported the research of the third author.

\section{The form of factors} \label{symmetric}

We begin with some very general results on factorization of iterates, culminating in Theorem \ref{symmetric thm}, which gives a generalization of \cite[Proposition 2.6]{settled} to arbitrary characteristic. %We require that the ground field $K$ have characteristic zero only to avoid complications involving inseparable extensions. 
The key lemma, Lemma \ref{surjective}, can be proven using Capelli's Lemma, as in \cite[Proposition 2.6]{settled}, but we give here a self-contained proof. 

Fix an algebraic closure $\overline{K}$ and separable closure $K^{\text{sep}}$ of $K$. Recall that $f \in K[x]$ is separable over $K$ if it has $\deg f$ distinct roots in $\overline{K}$, or equivalently, all roots of $f$ in $\overline{K}$ lie in $K^{\text{sep}}$. 

\begin{definition}
Let $K$ be a field and $f \in K[x]$ have degree $2n$ for some $n \geq 1$. We say that $f$ is \textit{symmetrically reducible} over $K$ if there exists a monic $g \in K[x]$ with $f(x) = Cg(x)g(-x)$ for some $C \in K$. % $h_1, h_2 \in K[x]$, each of degree $n$, with $f(x) = h_1(x)h_2(x)$ and $h_2(x) = h_1(-x)$.
%$$
%h_1(x) = \sum_{i = 0}^n a_ix^i, \qquad h_2(x) = \sum_{i=0}^n (-1)^ia_ix^i.
%$$
\end{definition}

\begin{lemma} \label{deglem}
Let $K$ be a field, and let $f(x) \in K[x]$ have degree $d \geq 2$. If $f^n(x)$ is newly reducible for some $n \geq 2$, then each irreducible factor of $f^n(x)$ has degree divisible by $d^{n-1}$.
\end{lemma}

\begin{proof}
Let $\alpha \in \overline{K}$ be a root of $f^n(x)$. Then $f(\alpha) \in K(\alpha)$, and so $K(f(\alpha)) \subseteq K(\alpha)$. Hence 
\begin{equation} \label{tower}
[K(\alpha) : K] = [K(\alpha) : K(f(\alpha))][K(f(\alpha)) : K] = d^{n-1}[K(\alpha) : K(f(\alpha))],
\end{equation}
where the last equality follows because $f(\alpha)$ is a root of $f^{n-1}(x)$ and $f^{n-1}$ is assumed to be irreducible over $K$. Now let $h$ be any irreducible factor of $f^n$, and $\alpha$ a root of $h$. Then $\deg h = [K(\alpha) : K]$ is divisible by $d^{n-1}$ from \eqref{tower}. 
\end{proof}

%[I'm in the middle of making the following lemma independent of characteristic. Need to be a bit careful. For instance, the absolute Galois group should be $\Gal(K^{\text{sep}}/K)$.]

\begin{lemma} \label{surjective}
Let $K$ be a field and let $g,f \in K[x]$ be separable over $K$. Assume $g$ is irreducible over $K$, and let $h \in K[x]$ be any non-constant factor of $g \circ f$. Consider the map
\begin{equation}
\Phi : \{\text{roots of $h$ in $K^{\text{sep}}$}\} \to \{\text{roots of $g$ in $K^{\text{sep}}$}\} \label{fmap}
\end{equation}
defined by $\Phi(\alpha) = f(\alpha)$. Then $\Phi$ is a surjective $k$-to-1 map for some $k$ with $1 \leq k \leq \deg f$.
\end{lemma}

\begin{proof}
First note that a root $\alpha$ of $h$ satisfies $g(f(\alpha)) = 0$, and so $f(\alpha)$ is indeed a root of $g$, showing that $\Phi$ is well-defined. Because $h$ is non-constant, the image of $\Phi$ is non-empty.

Let $G_K := \Gal(K^{\text{sep}}/K)$ be the absolute Galois group of $K$. The key observation is that since $f$ is defined over $K$, it must commute with the action of $G_K$ on $K^{\text{sep}}$. Hence if $\alpha'$ is a root of $g$ and $f(\alpha) = \alpha'$, then for any $\sigma \in G_K$ we have
\begin{equation} \label{galois}
\sigma(\alpha') = \sigma(f(\alpha)) = f(\sigma(\alpha)).
\end{equation}
Now let $\alpha', \beta'$ be roots of $g$. Because $g$ is irreducible over $K$, $G_K$ acts transitively on the roots of $g$, and hence there is some $\sigma \in G_K$ with $\sigma(\alpha') = \beta'$. From \eqref{galois} we have $\beta' = f(\sigma(\alpha))$. If we further assume that $\alpha$ is a root of $h$, then so must be $\sigma(\alpha)$, since $h$ is defined over $K$ and hence the set of its roots is preserved by $G_K$. Therefore $\sigma$ induces a map $\Phi^{-1}(\alpha') \to \Phi^{-1}(\beta')$. This map is injective since $\sigma$ is an injection from the set of all roots of $g \circ f$ into itself. Similarly, $\sigma^{-1}$ gives an injection $\Phi^{-1}(\beta') \to \Phi^{-1}(\alpha')$. It follows that $\#\Phi^{-1}(\alpha') = \#\Phi^{-1}(\beta')$, and thus all fibers of $\Phi$ have equal cardinality. But $\Phi$ has non-empty image, and the lemma is proved.
\end{proof}

\begin{lemma} \label{seplem}
Let $K$ be a field and $f \in K[x]$. If $f$ is irreducible over $K$ and $f'$ is not identically zero, then all iterates of $f$ are separable over $K$.%If $f$ has non-zero derivative, then so do all iterates of $f$.
\end{lemma}

\begin{proof}
Fix an algebraic closure $\overline{K}$ of $K$, and recall the well-known fact that $g \in K[x]$ is separable over $K$ if and only if $g$ and $g'$ have no common roots in $\overline{K}$. We proceed by induction on $n$, the number of iterations of $f$. Let $n = 1$, and suppose there is $\alpha \in \overline{K}$ with $f(\alpha) = f'(\alpha) = 0$. Because $f$ is irreducible over $K$, we have $[K(\alpha) : K] = \deg f$. But $\deg f' < \deg f$, and so $f'(\alpha) = 0$ forces $f'$ to be identically zero, which is a contradiction.  

Assume now that $f^{n}$ is separable over $K$ for some $n \geq 1$. The chain rule gives
\begin{equation} \label{chain}
(f^{n+1})'(x) = (f^{n})'(f(x)) \cdot f'(x).
\end{equation}
Suppose that $\alpha \in \overline{K}$ satisfies $f^{n+1}(\alpha) = (f^{n+1})'(\alpha) = 0$. From \eqref{chain}, we have either $(f^{n})'(f(\alpha)) = 0$ or $f'(\alpha) = 0$. In the former case, we also have $0 = f^{n+1}(\alpha) = f^{n}(f(\alpha))$, contradicting the separability of $f^n$. In the latter case, we also have $0 = f^{n+1}(\alpha) = f(f^n(\alpha))$, implying that $f^n(\alpha)$ is a root of $f$ with $[K(f^n(\alpha)) : K] \leq [K(\alpha) : K] \leq \deg f' < \deg f$. This contradicts the irreducibility of $f$. 
\end{proof}

\begin{theorem} \label{symmetric thm}
Let $K$ be a field, let $f(x) = ax^2 + bx + c \in K[x]$ with $a \neq 0$, and assume that $f'$ is non-zero. If $f^n(x)$ is newly reducible over $K$ for some $n \geq 2$, then there is a monic irreducible $h \in K[x]$ of degree $2^{n-1}$ such that 
$$f^n(x) = a^{2^n-1}h(x)h(-(x+(b/a))).$$
In particular, if $K$ has characteristic not equal to $2$, then $f^n(x + \gamma)$ is symmetrically reducible over $K$, where $\gamma = -b/2a$ is the critical point of $f$. 
\end{theorem}

\begin{proof}
Assume that $f^n$ is newly reducible over $K$ for some $n \geq 2$, and observe that by Lemma \ref{deglem} we have $f^n(x) = \ell(f^n)h_1(x)h_2(x)$, where $\ell(f^n)$ is the leading coefficient of $f^n$, and $h_1, h_2 \in K[x]$ are monic irreducibles of degree $2^{n-1}$. An easy induction shows that $\ell(f^n) = a^{2^n-1}$. By Lemma \ref{seplem}, both $f^{n-1}$ and $f^n$ are separable over $K$, and it follows that $h_1$ and $h_2$ are separable over $K$ as well. By Lemma \ref{surjective}, we have that for $i=1,2$ the maps
  \begin{equation*}
    \Phi_i : \{\text{roots of $h_i$ in $K^{\text{sep}}$}\} \to \{\text{roots of $f^{n-1}$ in $K^{\text{sep}}$}\}
  \end{equation*}
  defined by $\Phi_i(\alpha)=f(\alpha)$, are bijections. Therefore $\Phi_2^{-1} \circ \Phi_1$ is a bijection from the roots of $h_1$ to the roots of $h_2$. Letting $\alpha$ be a root of $h_1$, we have that $\Phi_2^{-1}(\Phi_1(\alpha))$ is a root $\beta$ of $h_2$ satisfying $f(\alpha) = f(\beta)$. This gives 
  $a\alpha^2 + b \alpha = a \beta^2 + b \beta$, and because $\alpha \neq \beta$ we further deduce $\beta = - \alpha - \frac{b}{a}$.

  Now let $\alpha_1, \ldots, \alpha_{2^{n-1}}$ be the roots of $h_1$ in $K^{\text{sep}}$ and $\beta_1, \ldots, \beta_{2^{n-1}}$ be the roots of $h_2$ in $K^{\text{sep}}$. We have
  $$
  h_2(x) = \prod_{i=1}^{2^{n-1}} (x - \beta_i) = \prod_{i=1}^{2^{n-1}} (x - (-\alpha_i - (b/a))) = \prod_{i=1}^{2^{n-1}} ((x + (b/a)) + \alpha_i). 
  $$
  Because $n \geq 1$, this last expression is the same as $\prod_{i=1}^{2^{n-1}} (-(x + \frac{b}{a}) - \alpha_i)$, which is $h_1(-(x+\frac{b}{a}))$. 
  
  If $K$ has characteristic not equal to $2$, then from $f^n(x) = \ell(f^n)h_1(x)h_1(-(x+\frac{b}{a}))$ we have $f^n(x - \frac{b}{2a}) = \ell(f^n)h_1(x-\frac{b}{2a})h_1(-(x+\frac{b}{2a})),$ whence 
  $f^n(x - \frac{b}{2a}) = \ell(f^n)g(x)g(-x)$ for $g(x) = h_1(x-\frac{b}{2a})$. Thus $f^n(x - \frac{b}{2a})$ is symmetrically reducible over $K$. 
  \end{proof}
  
  %That is, the bijection
  %\begin{equation*}
  %  \Psi : \{\text{roots of $h_1$ in $K^{\text{sep}}$}\} \to \{\text{roots of $h_2$ in $K^{\text{sep}}$}\}
  %\end{equation*}
  %defined by $\Psi=\Phi_2^{-1}\cdot\Phi_1$, is such that $\Psi(\alpha)=-\alpha-a$.
  %Equivalently, $\Psi(\gamma+\beta)=\gamma-\beta$.
  %Hence
  %\begin{equation*}
   % h_1(x)=\prod_{i=1}^{2^{n-1}}(x-(\gamma+\beta_i))
  %  \qquad\text{ and }\qquad
  %  h_2(x)=\prod_{i=1}^{2^{n-1}}(x-(\gamma-\beta_i))
  %\end{equation*}
 
\section{Quadratic polynomials with newly reducible second and third iterate} \label{surface}

%In this section we will study monic quadratic rational polynomials. 
A monic quadratic polynomial in $K[x]$ has the form $x^2+ax+b$, but when $K$ has characteristic not equal to 2 we may write it as $(x-\gamma)^2+m+\gamma$ with $\gamma = -a/2$ and $m = b + a/2 - a^2/4$. This latter form emphasizes the critical point $\gamma$ of $f$, and behaves more simply with respect to iteration. 
In this section, we prove Theorem \ref{characterization} and give some results in the case where $K$ has characteristic 2. 

To prove Theorem \ref{characterization}, it is enough, in light of Theorem \ref{symmetric thm}, to characterize quadratic polynomials with symmetrically reducible third iterate, and then discard those with a reducible first or second iterate. Clearly $f(x)$ is reducible over $K$ if and only if $\sqrt{-m-\gamma}$ is in $K$.

%\begin{lemma}\label{first-iterate-criterion}
%Let $f\left(x\right)=\left(x-\gamma\right)^2+\gamma+m\in\Q[x]$. Then $f\left(x\right)$ is reducible if and only if $\sqrt{-m-\gamma}\in\Q$.
%\end{lemma}
%\begin{proof}
%Over $\C$, $f(x)$ factors uniquely into monic factors as $(x-\gamma + \sqrt{-m-\gamma})(x-\gamma - \sqrt{-m-\gamma})$. These factors are rational polynomials if and only if $\sqrt{-m-\gamma}\in\Q$.
%\end{proof}

We begin by giving a criterion for when $f^2(x+\gamma)$ has symmetrically reducible second iterate. 

%To tell when the second iterate of $f\left(x\right)$ is reducible, the following criterion will be useful. The proof is straightforward, but it illustrates the method we will use to prove Theorem \ref{characterization}.
\begin{lemma}\label{second-iterate-criterion}
Let $K$ be a field of characteristic not equal to $2$, and let $f\left(x\right)=\left(x-\gamma\right)^2+\gamma+m$ for $\gamma, m\in K$. Then $f^2\left(x+\gamma\right)$ is symmetrically reducible over $K$ if and only if at least one of 
\begin{align*}
\sqrt{-2m\pm2\sqrt{m^2+m+\gamma}}% \qquad \sqrt{2m-2\sqrt{m^2+m+\gamma}}
\end{align*}
is in $K$.
\end{lemma}

\begin{proof}
%For all $f\left(x\right)=\left(x-\gamma\right)^2+\gamma+m\in\Q[x]$,
We have
$$f^2\left(x\right) = \left(x-\gamma\right)^4 +2 m \left(x-\gamma \right)^2 + m^2 + m + \gamma.$$
On the other hand, by definition $f^2\left(x+\gamma\right)$ is symmetrically reducible over $K$ if and only if $f^2\left(x+\gamma\right) = (x^2 + cx + d)(x^2 - cx + d)$ for some $c,d \in K$. 
%
%there exist $
%$h_1,h_2\in K[x]$ with $f^2\left(x+\gamma\right)=h_1\left(x+\gamma\right)h_2\left(x+\gamma\right)$ and
%\begin{align*}
%h_1\left(x+\gamma \right) &= x^2+bx+a, \\
%h_2\left(x+\gamma \right) &= x^2-bx+a.
%\end{align*}
Equating coefficients gives %of powers of $x$ gives%in the equation $f^2\left(x+\gamma\right)=h_1\left(x+\gamma\right)h_2\left(x+\gamma\right)$ gives
\begin{align}
m^2+m+\gamma &= d^2 \label{one}  \\
2m &= 2d-c^2 \label{two}
\end{align}
Clearly \eqref{one} has a solution $d \in K$ if and only if $\sqrt{m^2+m+\gamma} \in K$. Substituting $d = \pm \sqrt{m^2+m+\gamma}$ into \eqref{two}, we see that \begin{align*}
c &= \pm \sqrt{-2m\pm 2\sqrt{m^2+m+\gamma}}.
\end{align*}
If one of these choices of $c$ lies in $K$, then so does the corresponding choice of $d$, and hence $f^2\left(x+\gamma\right)$ is symmetrically reducible over $K$. 
On the other hand, if neither of $\sqrt{-2m\pm2\sqrt{m^2+m+\gamma}}$ is in $K$, then there is no $c \in K$ satisfying \eqref{two}, and hence $f^2(x + \gamma)$ is not symmetrically reducible over $K$. 
\end{proof}

\begin{proposition} \label{n22}
Let $K$ be a field of characteristic not equal to $2$, and let $K^2 = \{k^2 : k \in K\}$. Then $K \in \mathcal{N}_{2,2}$ (resp. $K \in \mathcal{N}_{2,2}^\infty$) if and only if $K \neq K^{2}$ (resp. $K \setminus K^{2}$ is infinite).
\end{proposition}

\begin{proof}
Equations \eqref{one} and \eqref{two} imply that for any $a,b \in K$ we can construct $f$ with symmetrically reducible second iterate, simply by taking $m = a - \frac{1}{2}b^2$ and $\gamma = a^2 - m^2 - m$. Taking $b = 2$ yields the family $f(x) = (x - (3a-2))^2 + 4a - 4$ for $a \in K$, for which we have 
\begin{equation*}
f^2(x) = \left(x^2 + (-6a + 2)x + 9a^2 - 5a \right)\left(x^2 + (-6a + 6)x + 9a^2 - 17a + 8\right).
\end{equation*}
Thus $f$ has newly reducible second iterate provided $f(x)$ is irreducible over $K$, and because $K$ has characteristic not equal to $2$, this is equivalent to $1-a \not\in K^2$. Because $x \mapsto 1-x$ is a bijection on $K$, it follows that $K \in \mathcal{N}_{2,2}$ (resp. $K \in \mathcal{N}_{2,2}^\infty$) if $K \neq K^{2}$ (resp. $K \setminus K^{2}$ is infinite). The converse follows from the observation that $K = K^2$ implies that no quadratic polynomial over $K$ is irreducible, and $K \setminus K^{2}$ finite implies that only finitely many quadratic polynomials over $K$ are irreducible. 
\end{proof}

When $K$ has characteristic 2, the corresponding result to Lemma \ref{second-iterate-criterion} is more complicated. 

\begin{lemma} \label{secondchar2}
Let $K$ be a field of characteristic 2, and let $f(x) = x^2 + ax + b$. Then $f^2(x) = h(x)h(x+a)$ for some $h \in K[x]$ if and only if $b = 0$, $b = a+1$, or
\begin{equation} \label{poly}
x^4 + a^2x^3 + a^3x^2 + a^2(b^2 + ab + b)x + (b^2 + ab + b)^2
\end{equation}
has a root in $K$.
\end{lemma}

\begin{proof}
We have $f^2(x) = x^4 + (a^2 + a)x^2 + a^2x + (b^2 + ab + b)$. Letting $h(x) = x^2 + cx + d$ and equating coefficients in $f^2(x) = h(x)h(x+a)$ gives
\begin{align*}
a^2 + ac + c^2 & = a^2 + a \\
a^2c + ac^2 &= a^2 \\
a^2d + acd + d^2 &= b^2 + ab + b
\end{align*}
The first equation gives $ac + c^2 + a = 0$, which renders the second equation redundant. Multiplying through by $a^2d^2$ gives $a^2d(acd) + (acd)^2 + a^3d^2 = 0$. From the third equation we have
$acd = a^2d + d^2 + b^2 + ab + b$, and substitution yields 
$$a^2d(a^2d + d^2 + b^2 + ab + b) + (a^2d + d^2 + b^2 + ab + b)^2 + a^3d^2 = 0.$$
Expanding, simplifying, and writing this as a polynomial in $d$ yields
$$d^4 + a^2d^3 + a^3d^2 + (a^3b + a^2b^2 + a^2b)d + a^2b^2 + b^4 + b^2 = 0,$$
which is equivalent to $x^4 + a^2x^3 + a^3x^2 + a^2(b^2 + ab + b)x + (b^2 + ab + b)^2$ having a root in $K$. When such a root exists, we may solve for $c$ provided that $a \neq 0$ and $d \neq 0$. But $a=0$ forces $c = 0$, and so we obtain $c$ anyway. If $d = 0$, then we get $b^2 + ab + b = 0$, and so $b = 0$ or $b = a+1$. 
\end{proof}

\begin{remark}
It follows from Theorem \ref{symmetric thm} and Lemma \ref{secondchar2} that $f(x)$ has a newly reducible second iterate over $K$ if and only if $f$ is irreducible and the polynomial in \eqref{poly} has a root in $K$. This is because if $b = 0$ then $f(0) = 0$, and if $b = a+1$ then $f(1) = 0$, and thus $f$ is reducible in both cases. 
\end{remark}

  % Observe that
  % \begin{align*}
  %   0 &= f\left(\gamma\pm\sqrt{-m-\gamma}\right) \\
  %   0 &= f^2\left(\gamma\pm\sqrt{-m\pm\sqrt{-m-\gamma}}\right)
  % \end{align*}

%\begin{proof}[Alternative Proof]
 % The roots of $f^2\left(x\right)$ are
 % $$\gamma\pm\sqrt{-m\pm\sqrt{-m-\gamma}}$$ % What is the best way to quickly show this?
 % Suppose $f^2\left(x+\gamma\right)$ is symmetrically reducible into $p_1\left(x+\gamma\right)p_2\left(x+\gamma\right)$.
 % Then and only then, for some choice of $\pm$, the following is rational:
  %\begin{align*}
   % p_1\left(x+\gamma\right)
    %  &= \left(x+\sqrt{-m+\sqrt{-m-\gamma}}\right)\left(x\pm\sqrt{-m-\sqrt{-m-\gamma}}\right) \\
    %  &= x^2+x\left(\sqrt{-m+\sqrt{-m-\gamma}}\pm\sqrt{-m-\sqrt{-m-\gamma}}\right)\pm\sqrt{m^2+m+\gamma} \\
    %  &= x^2+x\sqrt{\left(\sqrt{-m+\sqrt{-m-\gamma}}\pm\sqrt{-m-\sqrt{-m-\gamma}}\right)^2}\pm\sqrt{m^2+m+\gamma} \\
    %  &= x^2+x\sqrt{\left(-m+\sqrt{-m-\gamma}\right)\pm2\sqrt{m^2+m+\gamma}+\left(-m-\sqrt{-m-\gamma}\right)} \pm\sqrt{m^2+m+\gamma} \\
    %  &= x^2+x\sqrt{-2m\pm2\sqrt{m^2+m+\gamma}}\pm\sqrt{m^2+m+\gamma}
  %\end{align*}
%Then at least one of $\sqrt{-2m\pm2\sqrt{m^2+m+\gamma}}\in\Q[x]$.
%\end{proof}

We now turn to finite fields of characteristic 2. In particular, one might ask which of the fields $\mathbb{F}_{2^n}$ belong to $\mathcal{N}_{2,k}$ for various $k \geq 2$? When $k \geq 4$, this question has been settled by Ahmadi et al \cite[Theorem 10]{shparostafe} (see also \cite{ahmadi}), and the answer is surprising: none of them. The result \cite[Theorem 10]{shparostafe} gives that if $f \in \mathbb{F}_{2^n}[x]$ is quadratic, then $f^3(x)$ is reducible over $\mathbb{F}_{2^n}$. 

As for $\mathcal{N}_{2,2}$ and $\mathcal{N}_{2,3}$, observe that $\mathbb{F}_2 \not\in \mathcal{N}_{2,2}$, since the unique irreducible quadratic polynomial $f(x) = x^2 + x + 1 \in \mathbb{F}_2[x]$ has $f^2(x)$ irreducible. By adapting the methods of \cite{ahmadi}, we show the following. 
\begin{proposition} \label{fintwo}
Let $K = \mathbb{F}_{2^n}$. Then $K \in \mathcal{N}_{2,2}$ for all $n \geq 2$ and $K \in \mathcal{N}_{2,3}$ for all $n \geq 1$. 
\end{proposition}

\begin{proof}
Let $K = \mathbb{F}_{2^n}$, and $f(x) = x^2 + ax + b \in K[x]$, and denote the trace map $K \to \mathbb{F}_2$ by $\text{Tr}_{K/\mathbb{F}_2}$. By a standard result in field theory \cite[Corollary 3.6]{appfin}, we have that $f$ is irreducible over $K$ if and only if $a \neq 0$ and $\text{Tr}_{K/\mathbb{F}_2}(b/a^2) = 1$. Assume that this holds. Then by Capelli's Lemma (\cite{ahmadi, cohen2}), $f^2(x)$ is reducible over $K$ if and only if $f(x) - \alpha$ is reducible over $K(\alpha)$, where $f(\alpha) = 0$. Applying \cite[Corollary 3.6]{appfin} again, this is equivalent to $\text{Tr}_{K(\alpha)/\mathbb{F}_2}((b-\alpha)/a^2) = 0$. Properties of the trace now give
\begin{align*}
\text{Tr}_{K(\alpha)/\mathbb{F}_2}((b-\alpha)/a^2) &= \text{Tr}_{K(\alpha)/\mathbb{F}_2}(b/a^2) - \text{Tr}_{K(\alpha)/\mathbb{F}_2}(\alpha/a^2) \\
&= 2\text{Tr}_{K/\mathbb{F}_2}(b/a^2) + \text{Tr}_{K/\mathbb{F}_2}(\text{Tr}_{K(\alpha)/K}(\alpha/a^2)) \\
&= \text{Tr}_{K/\mathbb{F}_2}(a/a^2).
\end{align*}
To show that $K \in \mathcal{N}_{2,2}$, we thus seek $a,b \in K$ with $a \neq 0$, $\text{Tr}_{K/\mathbb{F}_2}(b/a^2) = 1$, and $\text{Tr}_{K/\mathbb{F}_2}(1/a) = 0$. If there are $r,s \in K$ with $r \neq 0$, $\text{Tr}_{K/\mathbb{F}_2}(r) = 0$, and $\text{Tr}_{K/\mathbb{F}_2}(s) = 1$, then we can take $a = 1/r$ and $b = s/r^2$. Now because $K$ is a separable extension of $\mathbb{F}_2$, the bilinear form $(x,y) = \text{Tr}_{K/\mathbb{F}_2}(xy)$ is non-degenerate, and taking $y=1$ gives that $\text{Tr}_{K/\mathbb{F}_2}$ is a surjective homomorphism from the additive group of $K$ to the additive group of $\mathbb{F}_2$. Therefore we can find the desired $r,s$ provided that $|K| > 2$. 

To show that $K \in \mathcal{N}_{2,3}$, we seek $a,b \in K$ with $a \neq 0$ and $\text{Tr}_{K/\mathbb{F}_2}(b/a^2) = \text{Tr}_{K/\mathbb{F}_2}(1/a) = 1$. This ensures that $f^2(x)$ is irreducible over $K$, and then by \cite[Theorem 10]{shparostafe} we have that $f^3(x)$ is reducible over $K$. If we find $r \in K$ with $r \neq 0$ and $\text{Tr}_{K/\mathbb{F}_2}(r) = 1$, then we may take $a = b = 1/r$. Because $\text{Tr}_{K/\mathbb{F}_2} : K^+ \to \mathbb{F}_2^+$ is a surjective homomorphism, the desired $r$ must exist. 
\end{proof}

The techniques of Proposition \ref{fintwo} and \cite[Theorem 10]{shparostafe} do not apply to general fields of characteristic 2, and it remains an open question which of them belong to $\mathcal{N}_{2,2}$ and $\mathcal{N}_{2,3}$. See Question \ref{charpquest}.

We now give a characterization of quadratic polynomials $f(x)$ such that $f^3\left(x+\gamma\right)$ is symmetrically reducible. 
%This will allow us to parametrize an infinite family of quadratics with newly reducible third iterates.
\begin{theorem} \label{symmchar}
Let $K$ be a field of characteristic not equal to 2, and let $f\left(x\right)=\left(x-\gamma\right)^2+\gamma+m\in K[x]$. Then $f^3\left(x+\gamma \right)$ is symmetrically reducible over $K$ if and only if $\gamma=-m$ or there exist $r,s\in K$ such that
\begin{align}
\gamma & = \frac{1}{256r^2}\left( -2 r^5 s^2+9 r^4 s^4-4 r^4 s^2-16 r^3 s^6+32 r^3 s^4+16 r^3 s^2+32 r^3+14 r^2 s^8\right. \label{gamma-in-r-s}\\ 
&\phantom{=\frac{1}{256r^2}(}-64 r^2 s^6-8 r^2 s^4+96 r^2 s^2+128 r^2-6 r s^{10}+48 r s^8-64 r s^6 \nonumber\\
&\phantom{=\frac{1}{256r^2}(}\left. -160 r s^4+96 r s^2+128 r+s^{12}-12 s^{10}+40 s^8-112 s^4-64 s^2 \right) \nonumber\\
m & = \frac{-r^2-2 r s^2-4 r+s^4-4 s^2-4}{8 r} \label{m-in-r-s}
\end{align}
\end{theorem}
\begin{proof}
First, if $\gamma=-m$ then $f^3(x+\gamma) = \left(x^4 - 2\gamma x^2 + \gamma ^2-\gamma\right)^2$, so $f^3(x+\gamma)$ is symmetrically reducible. Otherwise, suppose that there exist $r,s\in K$ that satisfy \eqref{gamma-in-r-s} and \eqref{m-in-r-s}. Note that $r\neq 0$. Then we compute $f^3\left(x+\gamma\right)$ and substitute the expressions in $r$ and $s$ for $\gamma$ and $m$. This gives
\begin{align*}
f^3\left(x+\gamma\right) &= \frac{h\left(x\right)}{64 r^2}\cdot\frac{h\left(-x\right)}{64r^2},
\end{align*}
where
\begin{align*}
h\left(x\right) &= 64 r^2 x^4 -64 r^2 s x^3 + \left(-16 r^3-64 r^2+16 r s^4-64 r s^2-64 r\right)x^2 \\&\phantom{=} +\left(24 r^3 s+64 r^2 s-8 r s^5+32 r s^3+32 r s\right)x + \left(r^4-12 r^3 s^2+10 r^2 s^4 \right.\\&\phantom{=}\left. -24 r^2 s^2-8 r^2 -4 r s^6+16 r s^4+16 r s^2+s^8-8 s^6+8 s^4+32 s^2+16\right).
%h_2\left(x\right) &= 64 r^2 x^4 + 64 r^2 s x^3 + \left(-16 r^3-64 r^2+16 r s^4-64 r s^2-64 r\right)x^2 \\&\phantom{=} - \left(24 r^3 s+ 64 r^2 s -8 r s^5 +32 r s^3 +32 r s\right)x + \left(r^4-12 r^3 s^2+10 r^2 s^4 \right.\\&\phantom{=}\left. -24 r^2 s^2 -8 r^2 -4 r s^6+16 r s^4+16 r s^2+s^8-8 s^6+8 s^4+32 s^2+16 \right).
\end{align*}
Therefore, $f^3\left(x+\gamma\right)$ is symmetrically reducible.

Conversely, suppose that $f^3\left(x+\gamma\right)$ is symmetrically reducible. Our goal is to show either that $\gamma=-m$ or that \eqref{gamma-in-r-s} and \eqref{m-in-r-s} hold for some $r,s\in K$. We have%First, we compute $f^3\left(x\right)$.
\begin{align*}
f^3\left(x+\gamma\right) %&= \left(\left(\left(x-\gamma \right)^2+m\right)^2+m\right)^2+m+\gamma  \\
&= x^8 +4 m x^6 +\left(6 m^2+2 m\right)x^4 +\left(4 m^3+4 m^2\right)x^2 + (m^4 + 2 m^3 + m^2 + m + \gamma).
\end{align*}
Since $f^3\left(x+\gamma\right)$ is symmetrically reducible, by definition there exist $h_1,h_2\in K[x]$ such that
\begin{align*}
h_1\left(x+\gamma\right) &= x^4+dx^3+cx^2+bx+a, \\
h_2\left(x+\gamma\right) &= x^4-dx^3+cx^2-bx+a,
\end{align*}
and $f^3\left(x+\gamma\right)=h_1\left(x+\gamma\right)h_2\left(x+\gamma\right)$. Equating coefficients gives
\begin{align}
\gamma +m^4+2 m^3+m^2+m &= a^2 \label{f3eq1}\\
4 m^3+4 m^2 &= 2 a c-b^2  \\
6 m^2+2 m &= 2 a-2 b d+c^2 \label{f3eq3}\\
4 m &= 2 c-d^2 \label{f3eq4}
\end{align}
First, we solve \eqref{f3eq3} for $a$ and substitute the resulting value into the other equations. We then solve \eqref{f3eq4} for $c$, and substitute the resulting value into the other equations. This gives \small
\begin{align}
 \gamma +m^4+2 m^3+m^2+m &= \frac{1}{64} \left(64 b^2 d^2-16 b d^5-128 b d^3 m+128 b d m^2+128 b d m+d^8+16 d^6 m \nonumber
  \right.\\&\phantom{=\frac{1}{64}(6}\left. +48 d^4 m^2 -16 d^4 m-128 d^2 m^3 -128 d^2 m^2+64 m^4+128 m^3+64m^2\right) \label{f3eq5} \\
4\left(m^3+m^2\right) &= \frac{1}{8} \left(-8 b^2+8 b d^3+32 b d m-d^6-12 d^4 m-24 d^2 m^2+8 d^2 m+32 m^3+32 m^2\right) \label{f3eq6}
\end{align}
\normalsize
Next we solve \eqref{f3eq6} for $b$, giving
\begin{align}\label{b-eq}
b &= \frac{1}{4} \left(2 d^3 + 8 d m \pm d \beta \right) \quad \text{where} \quad \beta = \sqrt{2 d^{4} + 8 d^{2} m + 16 m^{2} + 16 m}.
\end{align}
Finally, substituting \eqref{b-eq} into \eqref{f3eq5} and solving for $\gamma$ gives 
\begin{equation}\label{gamma}
\gamma=\pm \beta A + \frac{17}{64}d^{8} + \frac{5}{4}md^6 + \frac{11}{4}  m^{2}d^{4} + \frac{7}{4} md^{4} + 2 m^{3}d^{2}  + 2 m^{2}d^{2}  - m,
\end{equation}
where
\begin{equation*}
A = \frac{1}{16} d^2 \left(3 d^4+8 d^2 m+8 m^2+8 m\right).
\end{equation*}
Since $b \in K$, \eqref{b-eq} shows that $d=0$ or $\beta \in K$. If $d=0$, then \eqref{gamma} gives $\gamma = -m$ and we're done.
% Since $\gamma$ is rational, $A=0$ or $\beta\in\Q$. \\
% First, suppose that $A=0$. Then we have
% \begin{align*}
% -m-\gamma &= -m-\gamma + \left( 4m + \frac{7}{2}d^2\right) A \\
% &= -m-\left(\frac{17 d^{8}}{64} + \frac{5 m}{4} d^{6} + \frac{11 d^{4}}{4} m^{2} + \frac{7 m}{4} d^{4} + 2 d^{2} m^{3} + 2 d^{2} m^{2} - m\right) + \left( 4m + \frac{7}{2}d^2\right) A \\
% &= \frac{25 d^8}{64}+\frac{5 d^6 m}{4}+d^4 m^2 \\
% &= \frac{1}{64} d^4 \left(5 d^2+8 m\right)^2
% \end{align*}
% Miraculously, $-m-\gamma$ is a square! By Lemma \ref{first-iterate-criterion}, this means that $f\left(x\right)$ is reducible, and this contradicts our hypothesis. Therefore, $A\neq 0$.
Otherwise, $\beta\in K$. Then \eqref{b-eq} gives a $K$-rational point on the surface
\begin{align}
S: y^2 &= 2s^4 + 8s^2m + 16m^2 + 16m. \label{mainsurface}
%&= 16 m^2 + \left(8s^2+16\right)m + 2s^4 \nonumber
\end{align}
For fixed $s$, this equation is a conic in $y$ and $m$, and we can use rational projection to parametrize its $K$-rational points. The homogeneous form of $S$ can be written
\begin{align*}
\overline{S}: Y^2 &= 16 M^2 + \left(8 s^2 + 16\right) M Z + 2s^4 Z^2.
\end{align*}
Note that the rational point $P = [M:Y:Z] = [1:4:0]\in\mathbb{P}^2(K)$ is on $\overline{S}$; this is the point we will project through. Let $r_0\in K$ be arbitrary. The affine part of the line through $r_0$ and $P$ is given by
$y = 4 m + r_0$. To solve for the other intersection point, we substitute into \eqref{mainsurface} to get 
$
\left(4 m + r_0\right)^2 = 16 m^2 + \left(8s^2+16\right)m + 2s^4.
$
This equation is linear in $m$, and we get
\begin{align*}
m &= \frac{2 s^4-r_0^2}{8 \left(r_0-s^2-2\right)}.
\end{align*}
Note that $r_0-s^2-2=0$ corresponds to the line intersecting the point at infinity with multiplicity 2, so with this projection we don't miss any affine rational points. Taking $r=r_0-s^2-2$ now gives
\begin{align*}
m = m\left(r,s\right) &= \frac{2s^4-r^2-2 r s^2-4 r-s^4-4 s^2-4}{8r}
\end{align*}
This is \eqref{m-in-r-s}, so we're halfway done. Since we have $y=4m+r_0=4m+r+s^2+2$,
\begin{align*}
y = y\left(r,s\right) &= \frac{r^2+s^4-4 s^2-4}{2 r}
\end{align*}
Recall that $y=\beta$ and $s=d$. Using these in \eqref{gamma},
\begin{align*}
\gamma=\pm y A + \frac{17}{64}s^{8} + \frac{5}{4}ms^6 + \frac{11}{4}  m^{2}s^{4} + \frac{7}{4} ms^{4} + 2 m^{3}s^{2}  + 2 m^{2}s^{2}  - m.
\end{align*}
If the plus sign satisfies this equation, it simplifies to
\begin{align*}
\gamma = \gamma\left(r,s\right) & = \frac{1}{256r^2}\left(-2 r^5 s^2+9 r^4 s^4-4 r^4 s^2-16 r^3 s^6+32 r^3 s^4+16 r^3 s^2+32 r^3+14 r^2 s^8 \right. \\ 
&\phantom{=\frac{1}{256r^2}(}-64 r^2 s^6-8 r^2 s^4+96 r^2 s^2+128 r^2-6 r s^{10}+48 r s^8-64 r s^6 \nonumber\\
&\phantom{=\frac{1}{256r^2}(}\left. -160 r s^4+96 r s^2+128 r+s^{12}-12 s^{10}+40 s^8-112 s^4-64 s^2\right). \nonumber
\end{align*}
This is \eqref{gamma-in-r-s}. On the other hand, if we must choose the minus sign, let
\begin{align*}
r' &= \frac{-s^4+4 s^2+4}{r} \\ s' &= s
\end{align*}
Plugging these in, $m\left(r',s'\right)=m\left(r,s\right)$ and $y\left(r',s'\right)=-y\left(r,s\right)$. So in this case, \eqref{gamma-in-r-s} is satisfied by $r'$ and $s'$.
\end{proof}

We make a few observations about Theorem \ref{symmchar} and its proof. First, Theorem \ref{characterization} is now proved, as a consequence of Theorem \ref{symmetric thm} and Theorems \ref{second-iterate-criterion} and \ref{symmchar}. Second, whenever $\gamma=-m$, we have $f(x)=(x-\gamma)^2$, which is obviously reducible. Therefore, Theorem \ref{symmchar} shows that all $f(x)$ with newly reducible third iterate have $m,\gamma$ such that \eqref{gamma-in-r-s} and \eqref{m-in-r-s} are satisfied by some $r,s\in K$. 

Third, we emphasize that $K$-rational points on the surface $S$ described in \eqref{mainsurface} play a crucial role in determining the existence of quadratic polynomials over $K$ with newly reducible third iterate. The fibers of the map $S \to \mathbb{A}^1$ given by projection onto the $m$-coordinate are particularly interesting, and worth dwelling on for a moment. In general, these fibers are elliptic curves, unless $m$ divides the discriminant of $2s^4 + 8s^2m + 16m^2 + 16m$ considered as a polynomial in $s$. Observe that
$$
\Disc(2s^4 + 8s^2m + 16m^2 + 16m) = 2^{21}m^3(m+2)^2(m+1).
$$
Hence we obtain a curve of genus zero if and only if $m \in \{0, -1, -2\}$. We remark that for each of these $m$-values, $f(x) = (x-\gamma)^2 + \gamma + m$ is post-critically finite -- that is, the forward orbit of $\gamma$ is finite -- and indeed these are the only such $m$-values for $K = \Q$. 

When $m = 0$, we have that $S$ degenerates to $y^2 = 2s^4$, which has no $K$-rational points unless $\sqrt{2} \in K$. Hence if $\sqrt{2} \not\in K$, then $f(x) = (x - \gamma)^2 + \gamma$ has $f^3(x + \gamma)$ symmetrically reducible only if $\gamma = -m$, which implies that $f(x)$ is reducible. If $\sqrt{2} \in K$, then equations \eqref{f3eq1} - \eqref{f3eq4} yield
$
\gamma = d^8 \left( \frac{17}{64} \pm \frac{3 \sqrt{2}}{16} \right),
$
and it follows that 
$$\sqrt{2\sqrt{\gamma}} = \frac{d^2}{2}(1 + \sqrt{2}) \in K.$$
By Lemma \ref{second-iterate-criterion} this shows that $f^2(x)$ is reducible over $K$. This is in line with \cite[Theorem 1.2(i) and Remark 4.3]{goksel}, which give that when $K = \Q$ and $m = 0$, either $f^2(x)$ is reducible over $K$, or all iterates of $f$ are irreducible over $K$. %and so we get an infinite family of maps of othe form  
% $(x - \gamma)^2 + \gamma$ with newly reducible third iterate by taking [Need to add assumptions that ensure that $f$ and $f^2$ are irreducible in this case].

When $m = -2$, $S$ degenerates to $y^2 = 2(s^2-2)^2$, and again one finds that $f^3(x+\gamma)$ is symmetrically reducible if and only if either $f(x)$ or $f^2(x)$ is reducible. 

When $m = -1$, matters are different, and the following proposition shows we obtain an infinite family of polynomials with newly reducible third iterate for certain $K$. 

\begin{proposition} \label{m-1}
Let $K$ be a field of characteristic not equal to 2, and let $f(x) = (x-\gamma)^2 + \gamma - 1$. Then $f^3(x+\gamma)$ is symmetrically reducible if and only if
\begin{equation} \label{gammaeq}
\gamma = 4\left(\frac{(t^2 - 8t + 8)(t^2 + 8)}{(t^2 - 8)^2}\right)^4 + 1 \quad \text{for some $t \in K$.}
\end{equation}
Suppose moreover that $K$ is a field with a non-trivial discrete valuation $v$ such that $v(5)$ is odd, and that $-1$ is not a square in $K$. 
%number field with ring of integers $\mathcal{O}_K$, with $-1$ not a square in $K$ and the ideal $(5)$ not the square of an ideal in $\mathcal{O}_K$. 
Then $f^3(x)$ is newly reducible over $K$ whenever $t = 25r$ in \eqref{gammaeq}, for any $r \in K$ with $v(r) \geq 0$. In particular, $K \in \mathcal{N}_{2,3}^\infty$. 
\end{proposition}

\begin{proof}

When $m = -1$, $S$ degenerates to $y^2 = 2s^2(s^2 - 4)$, whose solutions are given by $s = -2\frac{t^2 + 8}{t^2-8}$ where $t$ is any element of $K$. In the notation of the proof of Theorem \ref{symmchar} we have $s = d$, and we now use equations \eqref{f3eq1} - \eqref{f3eq4} to find \eqref{gammaeq}.

 Observe that for $m = -1$ we have $-m - \gamma = 1-\gamma$, and from \eqref{gammaeq} this is not a square in $K$ because $-1$ is not a square in $K$. We also have $m^2 + m + \gamma = \gamma$, and so if $\gamma$ is not a square in $K$, then Lemma \ref{second-iterate-criterion} gives that $f^2(x)$ is irreducible over $K$. Now from \eqref{gammaeq} we have that the numerator of $\gamma$ as a function of $t$ has constant term $c=5 \cdot 8^8$. Let $v : K^* \to \Z$ be a normalized discrete valuation with $v(5)$ odd. Then $v(c) = v(5) + 8v(8)$, and so $v(c)$ is odd. Taking $t = 25r$ with $v(r) \geq 0$ then shows $v(\gamma)$ is odd, and hence $\gamma$ cannot be a square in $K$. 
 % . It follows that if $K$ is a number field and the ideal $(5)$ is not a square in the ring of integers $\mathcal{O}_K$, then we may take $t = 25r$ with $r \in \mathcal{O}_K$ to get a family of maps $f(x) = (x-\gamma)^2 + \gamma - 1$ defined over $K$ having newly reducible third iterate. 
\end{proof}

\begin{remark}
Proposition \ref{m-1} shows that if $K$ is a number field with ring of integers $\mathcal{O}_K$, $-1$ is not a square in $K$, and the ideal $(5)$ is not the square of an ideal in $\mathcal{O}_K$, then $K \in \mathcal{N}_{2,3}^\infty$. In particular, $\Q \in \mathcal{N}_{2,3}^\infty$. 
\end{remark}

The family appearing in Proposition \ref{m-1} is the same family that appears in Lemma 3.9 of \cite{goksel}, though there it is stated only for integer parameters.

In the preceding discussion, we took $m$ to be constant. %In the case $K = \Q$ we found that only $m = -1$ gives infinitely many maps with newly reducible third iterate. 
It is of interest to give another infinite family with non-constant $m$-values, which we do in the following corollary to Theorem \ref{symmchar}. 

\begin{corollary}\label{newfamily}
Let $K$ be a totally real number field and suppose there is a prime $\p$ of $\mathcal{O}_K$ lying over the ideal $(3)$ with residue degree 1.  Let $f(x) = (x - \gamma)^2 + \gamma + m$, where 
\begin{align*}
m &=  -1 - 2 k^2 + k^4, \\
\gamma &= 1 + k^2 \left(-2 + k^2\right) \left(-1 - 4 k^2 + 2 k^4\right) \left(1 - 4 k^2 + 2 k^4\right)
\end{align*}
for $k \in \mathcal{O}_K$. If $|k| > \sqrt{2}$ and $\p \nmid (k)$, then $f^3(x)$ has a newly reducible third iterate over $K$. In particular, $K \in \mathcal{N}_{2,3}^\infty$.
\end{corollary}
\begin{proof}
Let $r=2$ and $s=2k$ in Theorem \ref{symmchar}, and observe that it's sufficient to show that neither of $-m-\gamma$ and $m^2 + m + \gamma$ is a square in $K$. One calculates $-m-\gamma = -4 k^6 \left(-2 + k^2\right)^3$.
Because $|k| > \sqrt{2}$, we have $-m-\gamma < 0$, and since $K$ is totally real, we have that $-m-\gamma$ is not a square in $K$. 
%, so $\sqrt{-m-\gamma}\not\in\Q$, proving that $f\left(x\right)$ is irreducible. 
We also have
\begin{align*}
m^2 + m + \gamma &= \left(-1 + k^2\right)^2 \left(1 + 6 k^2 + 13 k^4 - 16 k^6 + 4 k^8\right)
\end{align*}
%This is a square if and only if $\left(1 + 6 k^2 + 13 k^4 - 16 k^6 + 4 k^8\right)$ is a square. Since 
Observe that because $\p$ has residue degree 1, we have $\mathcal{O}_K/\p \cong \Z / 3\Z$. Because $\p \nmid (k)$ we conclude $k^2 \equiv 1 \pmod{\p}$, whence
\begin{equation*}
1 + 6 k^2 + 13 k^4 - 16 k^6 + 4 k^8 \equiv 2 \pmod{\p},
\end{equation*}
proving that $m^2 + m + \gamma$ is not a square in $\mathcal{O}_K$, and hence not a square in $K$.%, and by Lemma \ref{second-iterate-criterion} we conclude that $f^2(x)$ is irreducible over $\Q$. 
% thus %Therefore, since an integer square is either 0 or 1 mod 3, $\sqrt{m^2+m+\gamma}\not\in\Q$. 
%neither of $\sqrt{2m\pm2\sqrt{m^2+m+\gamma}}$ is rational. Hence by Lemma \ref{second-iterate-criterion}, $f^2\left(x\right)$ is irreducible, and thus $f^3\left(x\right)$ is newly reducible. %This holds for any integer $k$ that is 1 or 2 mod 3 (except -1 and 1), so there are infinitely many monic quadratic $f\in\Z[x]$ with newly reducible third iterate.
\end{proof}
Note that many other infinite families exist; for instance, there is a similar one with $r=-2$. 
%However, for each family there are some ground fields $K$ over which either $f(x)$ or $f^2(x)$ becomes reducible. 
%This leads us to make the following conjecture. 

%\begin{conjecture}
%Let $K$ be a finite extension of $\Q$ or $\Fp(t)$. Then $K \in \mathcal{N}_{2,3}$ and $K \in \mathcal{N}_{d,2}$ for every $d \geq 2$. %there exist infinitely many quadratic $f \in K[x]$ with newly reducible third iterate.  
%\end{conjecture}

%In general, the monic quadratic polynomials with newly reducible third iterate are exactly those $f(x)=\left(x-\gamma\right)^2+m+\gamma$ such that $\sqrt{-m-\gamma}\not\in\Q$, $\sqrt{2m + 2\sqrt{m^2+m+\gamma}}\not\in\Q$, $\sqrt{2m - 2\sqrt{m^2+m+\gamma}}\not\in\Q$, and there exist $r,s\in\Q$ such that \eqref{gamma-in-r-s} and \eqref{m-in-r-s} are satisfied.

\section{Rational Quadratics with newly reducible fourth iterate} \label{fourth}

In this section we briefly examine the following question.
\begin{question}
Does there exist $f \in \Q[x]$ of degree $2$ with newly reducible fourth iterate?
\end{question}

Theorem \ref{symmetric thm} tells us that if $f(x)=(x-\gamma)^2+m+\gamma$ has a newly reducible fourth iterate, then $f^4(x+\gamma)$ is symmetrically reducible. That is, $f^4(x+\gamma) = h(x)h(-x),$ where
\begin{align*}
h(x) &= x^8	+a_7 x^7	+a_6 x^6	+a_5 x^5	+a_4 x^4	+a_3 x^3	+a_2 x^2	+a_1 x	+a_0 
%h_2(x) &= x^8	-a_7 x^7	+a_6 x^6	-a_5 x^5	+a_4 x^4	-a_3 x^3	+a_2 x^2	-a_1 x	+a_0
\end{align*}
Equating coefficients gives the following system of equations:
\begin{align*}
\gamma +m^8+4 m^7+6 m^6+6 m^5+5 m^4+2 m^3+m^2+m &= a_0^2 \\ 
8 m^7+24 m^6+24 m^5+16 m^4+8 m^3 &= 2 a_0 a_2-a_1^2 \\ 
28 m^6+60 m^5+36 m^4+16 m^3+4 m^2 &= a_2^2-2 a_1 a_3+2 a_0 a_4 \\ 
56 m^5+80 m^4+24 m^3+8 m^2 &= -a_3^2+2 a_2 a_4-2 a_1 a_5+2 a_0 a_6 \\ 
70 m^4+60 m^3+6 m^2+2 m &= a_4^2+2 a_0-2 a_3 a_5+2 a_2 a_6-2 a_1 a_7 \\ 
56 m^3+24 m^2 &= -a_5^2+2 a_2+2 a_4 a_6-2 a_3 a_7 \\ 
28 m^2+4 m &= a_6^2+2 a_4-2 a_5 a_7 \\ 
8 m &= 2 a_6-a_7^2
\end{align*}
This defines a surface, which we denote $S_4$, whose rational points correspond to quadratic $f(x) \in \Q[x]$ such that $f^4(x+\gamma)$ is symmetrically reducible over $\Q$. In particular, this includes all monic quadratic $f(x)$ with newly reducible fourth iterate over $\Q$. At present, obtaining any information about the rational points on $S_4$ seems to be quite difficult. 

It is interesting to consider the fibers of the map $\pi : S_4 \mapsto \mathbb{A}^1$ given by projection onto the $m$-coordinate. In general these fibers appear to be very high-genus curves, though it is possible that some particular fibers are sufficiently singular that the genus drops significantly. This was the case for the fibers $m \in \{0, -1, -2\}$ for the similar projection map of the curve $S$ in \eqref{mainsurface} that arose from studying $f$ with symmetrically reducible third iterate. Similarly, the fibers $\pi^{-1}(0)$ and $\pi^{-1}(-2)$ are likely to have smaller genera than other fibers; unfortunately, their rational points cannot lead to $f(x) \in \Q[x]$ with newly reducible fourth iterate (see \cite[Theorem 4.1 and Remark 4.3]{goksel}). 

The fiber $\pi^{-1}(-1)$, on the other hand, remains a possible location for rational points on $S_4$. Lemma 3.10 of \cite{goksel} shows that \textit{integer} values of $\gamma$ cannot lead to rational points on this fiber, but the case of non-integer $\gamma \in \Q$ remains open (see Remark 4.3 of \cite{goksel}). 

%One might hope to construct a family with $m=-1$ having newly reducible fourth iterates, similar to Proposition \ref{m-1}, but \cite[Lemma 3.10]{goksel} shows this is impossible. [Actually it doesn't, since it only deals with integer $\gamma$, and as noted in Remark 4.3, the method breaks for rational $\gamma$. So maybe somehow there's hope for $m=-1$!])

%So what's the genus of the curve $\pi^{-1}(-1)$? Would be good to know. 

%Note that the last five equations are linear in $a_1,a_0,a_2,a_4,$ and $a_6$ respectively. Solving for each of these and making substitutions reduces this to three equations:
%\begin{align}
%\gamma +m^8+4 m^7+6 m^6+6 m^5+5 m^4+2 m^3+m^2+m &= \frac{A}{4096 \left(-8 a_7 m-a_7^3+2 a_5\right)^2}\label{fourth-iterate-surface-0}\\
%8 \left(m^7+3 m^6+3 m^5+2 m^4+m^3\right) &= \frac{B}{16384 \left(-8 a_7 m-a_7^3+2 a_5\right)^2}\label{fourth-iterate-surface-1}\\
%4 \left(7 m^6+15 m^5+9 m^4+4 m^3+m^2\right) &=\frac{C}{256 \left(-8 a_7 m-a_7^3+2 a_5\right)} \label{fourth-iterate-surface-2}
%\end{align}
%Here $A,B,$ and $C$ are very large expressions in $a_3,a_5,a_7,$ and $m$.
%Equation \eqref{fourth-iterate-surface-0} is linear in $\gamma$, which does not appear in \eqref{fourth-iterate-surface-1} and \eqref{fourth-iterate-surface-2}, so the problem is reduced to finding rational points on the surface defined by \eqref{fourth-iterate-surface-1} and \eqref{fourth-iterate-surface-2}.

\section{Cubics with newly reducible second iterate} \label{cubics}

In this section we study the question of whether there exist monic irreducible cubic polynomials in $K[x]$ with newly reducible second iterate. Observe that if $K$ does not have characteristic 3, then $f(x) = x^3 + c_2x^2 + c_1x + c_0$ can be conjugated to a cubic with no quadratic term: taking $\gamma = -c_2/3$, $a = c_1 - (c_2^2/3) + (c_2/3)$, and $b = c_0 - (c_2^3/27) -(c_2/3)$ gives 
 $f(x) = (x-\gamma)^3 + a(x-\gamma) + b + \gamma$, or equivalently
\begin{equation} \label{cubicform}
f(x + \gamma) = x^3 + ax + b + \gamma \in K[x]
\end{equation}
We emphasize that, unlike in the quadratic case, $\gamma$ is \textit{not} a critical point for $f$ except in the special case that $a=0$. It is this special case that is the source of our main result on newly reducible cubics:

\begin{theorem} \label{cubic2}
Let $K$ be a field of characteristic not equal to 2 or 3, and assume that $2$ is not the cube of an element of $K$. If $f(x) = (x-\gamma)^3 + b + \gamma$, with 
\begin{align} 
b = 36t^3 & \quad \text{and} \quad \gamma = (-2 \cdot 6^6)t^9 - 36t^3, \quad \text{or}  \label{firstfamily}\\
b = -9t^3& \quad \text{and} \quad  \gamma = 9t^3 + (2\cdot 3^6)t^9 \label{secondfamily}
\end{align}
for some non-zero $t \in K$, then $f$ has a newly reducible second iterate. 
\end{theorem}

\begin{proof}
From Lemma \ref{deglem}, we have that if $f^2(x)$ is newly reducible, then $f^2(x) =  p_1(x)p_2(x)$, where $\deg p_1(x) = 3$, $\deg p_2(x) = 6$, and $p_1(x)$ is irreducible. 
Write 
$$f^2(x + \gamma) = (x^3 + a_2x^2 + a_1x + a_0)(x^6 + b_5x^5 + b_4x^4 + b_3x^3 + b_2x^2 + b_1x + b_0),$$ and equate coefficients to get the system of equations
\begin{align*} 
b + b^3 + \gamma & =  a_0b_0 \\
0 & =  a_0b_1 + a_1b_0 \\
0 & =  a_0b_2 + a_1b_1 + a_2b_0 \\
3b^2 & =  a_0b_3 + a_1b_2 + a_2b_1 + b_0 \\
0 & =  a_0 b_4 + a_1 b_3 + a_2 b_2 + b_1 \\
0 & =  a_0b_5 + a_1 b_4 + a_2 b_3 + b_2 \\
3b & =  a_0 + a_1 b_5 + a_2 b_4 + b_3 \\
0 & =  a_1 + a_2 b_5 + b_4 \\
0 & =  a_2 + b_5
\end{align*}
We may solve for $b_5, b_4, b_3, b_2, b_1$, and $b_0$ to obtain
%\begin{align*}
%b + ab + b^3 + \gamma &=  a a_0 + a^3 a_0 + a_0^3 - 3 a^2 a_0 a_1 + 3 a a_0 a_1^2 - a_0 a_1^3 + 
% 6 a a_0^2 a_2 - 6 a_0^2 a_1 a_2 \\
% &\phantom{=} + 3 a^2 a_0 a_2^2 - 9 a a_0 a_1 a_2^2 + 
% 6 a_0 a_1^2 a_2^2 + 4 a_0^2 a_2^3 + 3 a a_0 a_2^4 - 5 a_0 a_1 a_2^4 + 
% a_0 a_2^6 \\
% &\phantom{=} - 3 a_0^2 b - 6 a a_0 a_2 b + 6 a_0 a_1 a_2 b - 3 a_0 a_2^3 b + 
% 3 a_0 b^2 \\
% a^2 + 3ab &=  -3 a a_0^2 + a a_1 + a^3 a_1 + 3 a_0^2 a_1 - 3 a^2 a_1^2 + 
% 3 a a_1^3 - a_1^4 - 3 a^2 a_0 a_2 + 12 a a_0 a_1 a_2 \\
% &\phantom{=} - 9 a_0 a_1^2 a_2 - 
% 3 a_0^2 a_2^2 + 3 a^2 a_1 a_2^2 - 9 a a_1^2 a_2^2 + 6 a_1^3 a_2^2 - 
% 3 a a_0 a_2^3 + 8 a_0 a_1 a_2^3 + 3 a a_1 a_2^4 \\
% &\phantom{=} - 5 a_1^2 a_2^4 - a_0 a_2^5 + 
% a_1 a_2^6 + 6 a a_0 b - 6 a_0 a_1 b - 6 a a_1 a_2 b + 6 a_1^2 a_2 b + 
% 3 a_0 a_2^2 b - 3 a_1 a_2^3 b + 3 a_1 b^2 \\
% 3a^2b &=  3 a^2 a_0 - 6 a a_0 a_1 + 3 a_0 a_1^2 + a a_2 + a^3 a_2 + 3 a_0^2 a_2 - 
% 6 a^2 a_1 a_2 + 9 a a_1^2 a_2 \\
% &\phantom{=} - 4 a_1^3 a_2 + 9 a a_0 a_2^2 - 
% 12 a_0 a_1 a_2^2 + 3 a^2 a_2^3 - 12 a a_1 a_2^3 + 10 a_1^2 a_2^3 + 
% 5 a_0 a_2^4 + 3 a a_2^5 - 6 a_1 a_2^5 \\
% &\phantom{=} + a_2^7 + 6 a a_1 b - 3 a_1^2 b - 
% 6 a_0 a_2 b - 6 a a_2^2 b + 9 a_1 a_2^2 b - 3 a_2^4 b + 3 a_2 b^2
%\end{align*}
%Considering only the case when $a = 0$ further simplifies the system to 
\begin{align}
b + b^3 + \gamma &=  a_0^3 - a_0 a_1^3 - 6 a_0^2 a_1 a_2 + 6 a_0 a_1^2 a_2^2 + 4 a_0^2 a_2^3 - 
 5 a_0 a_1 a_2^4 + a_0 a_2^6 \label{cubic_equation_1} \\
 &\phantom{=} - 3 a_0^2 b + 6 a_0 a_1 a_2 b - 3 a_0 a_2^3 b + 
 3 a_0 b^2 \nonumber \\
 0 &=  3 a_0^2 a_1 - a_1^4 - 9 a_0 a_1^2 a_2 - 3 a_0^2 a_2^2 + 6 a_1^3 a_2^2 + 
 8 a_0 a_1 a_2^3 - 5 a_1^2 a_2^4 \label{cubic_equation_2} \\
 &\phantom{=} - a_0 a_2^5 + a_1 a_2^6 - 6 a_0 a_1 b + 
 6 a_1^2 a_2 b + 3 a_0 a_2^2 b \nonumber \\
 &\phantom{=} - 3 a_1 a_2^3 b + 3 a_1 b^2 \nonumber \\
 0 &=  3 a_0 a_1^2 + 3 a_0^2 a_2 - 4 a_1^3 a_2 - 12 a_0 a_1 a_2^2 + 10 a_1^2 a_2^3 + 
 5 a_0 a_2^4 - 6 a_1 a_2^5 + a_2^7  \label{cubic_equation_3}\\
 &\phantom{=} - 3 a_1^2 b - 6 a_0 a_2 b + 9 a_1 a_2^2 b - 
 3 a_2^4 b + 3 a_2 b^2 \nonumber 
\end{align}

%At this point, further direct computation became unmanageable or impossible, and further specialization was necessary.  To simplify these equations, 
Note that \eqref{cubic_equation_2} is quadratic in $a_0$.  Taking the discriminant of this equation in $a_0$ gives
\begin{equation} \label{discriminant1}
9 a_2^4 b^2 + \left(36 a_1^3 a_2 - 42 a_1^2 a_2^3 + 24 a_1 a_2^5 - 
 6 a_2^7\right)b + 12 a_1^5 - 3 a_1^4 a_2^2 - 12 a_1^3 a_2^4 + 10 a_1^2 a_2^6 - 
 4 a_1 a_2^8 + a_2^{10}, 
\end{equation}
which in turn is quadratic in $b$.  The discriminant of this equation in $b$ is
\begin{align*}
& -144 \left(-9 a_1^6 a_2^2 + 24 a_1^5 a_2^4 - 25 a_1^4 a_2^6 + 14 a_1^3 a_2^8 - 
   5 a_1^2 a_2^{10} + a_1 a_2^{12}\right), 
\end{align*}   
which has the semi-miraculous factorization 
\begin{equation} \label{semimiracle}
12\left(a_1 a_2^2 \left(a_1 - a_2^2\right) \left(9 a_1^4 - 15 a_1^3 a_2^2 + 10 a_1^2 a_2^4 - 
   4 a_1 a_2^6 + a_2^8\right)\right).
\end{equation}
Selecting $a_1$ and $a_2$ so that \eqref{semimiracle} vanishes forces \eqref{discriminant1} to be a square, which gives a $K$-rational solution to \eqref{cubic_equation_2}.%, and dramatically simplifies the remaining equations. 

One such choice is $a_1 = 0$, which reduces \eqref{cubic_equation_2} to $0 = -3a_0^2 a_2^2 - a_0 a_2^5 + 3a_0a_2^2 b$, giving $a_0 = 0$ or $a_0 = -a_2^3/3 + b$. The former is impossible, since $a_0 = 0$ would imply that the degree three factor of our second iterate was reducible. Substituting $a_0 = -a_2^3/3 + b$ into \eqref{cubic_equation_3} produces another semi-miracle, as the $b^2$ terms cancel, giving $b = a_2^3/6$. Then \eqref{cubic_equation_1} gives $\gamma = \frac{-a_2^9}{108} - \frac{a_2^3}{6}$. Observe that $b+\gamma = -a_2^9/108$, and it follows from Proposition \ref{xd-c-fact} that $f(x)$ is irreducible over $K$ if and only if $a_2^9/108$ is not a cube in $K$. But this is equivalent to $2$ not being a cube in $K$, which is true by assumption. Hence $f$ has newly reducible second iterate over $K$. Taking $a_2 = 6t$ clears denominators and gives the family in \eqref{firstfamily}.

Taking $a_1 = a_2^2$ also causes \eqref{semimiracle} to vanish, and reduces \eqref{cubic_equation_2} to  
$$0 =   - 2 a_0 a_2^5 + a_2^8 - 3 a_0 a_2^2 b + 3 a_2^5 b + 3 a_2^2 b^2,$$
giving $a_0 = \frac{a_2^8  + 3 a_2^5 b  + 3 a_2^2 b^2}{2a_2^5 + 3 a_2^2 b}$. Substituting this into \eqref{cubic_equation_3} yields a cubic rational function in $b$ whose numerator fortuitously has a factor of $a_2^3 + 3b$, together with an irreducible quadratic in $b$. Taking $b = -a_2^3/3$ and substituting this into \eqref{cubic_equation_1} yields $\gamma = \frac{1}{27}(9a_2^3 + 2a_2^9)$. Then $b + \gamma = \frac{2}{27}a_2^3,$ and we have that $f$ is irreducible over $K$ since $-2$ is not a cube in $K$, as in the previous paragraph. Taking $a_2 = 3t$ gives the family in \eqref{secondfamily}.
\end{proof}

\begin{remark}
We can also force \eqref{semimiracle} to vanish by taking $a_2 = 0$, but this produces a family of polynomials $f$ that are all reducible. 
\end{remark}

%Notice that setting $a_1 = 0$ is not the only way to make \eqref{semimiracle} vanish.  This can also be achieved by setting $a_2 = 0$ or by setting $a_2 = a_1^2$.  The case $a_2 = a_1^2$ is similar to the case presented in the proof of Theorem \ref{cubic2}, and produces the family $b = \frac{-a_2^3}{3}$ and $\gamma = \frac{1}{27}(9a_2^3 + 2a_2^9)$ of monic cubics in $\Q[x]$ with newly reducible second iterate.  Again, we have a subfamily of monic cubics in $\Z[x]$.  In this case, we set $a_2 = 3t$ for some $t \in \Z[x]$, producing the family $b = -9t^3$ and $\gamma = 9t^3 + (2\cdot 3^6)t^9$.  The second iterate of cubics in this family has the following explicit factorization:
%\[f(x + 9t^3 + (2\cdot 3^6)t^9) = \frac{1}{27} \left(3x^3 + 3a_2 x^2 + 3a_2^2 x + a_2^3\right)\left(9x^6 - 9a_2x^5 -3a_2^3x^3 + 6a_2^4x^2 -3a_2^5x + a_2^6\right)\] Setting $a_2 = 0$, unfortunately causes the first iterate to be reducible, so there is no corresponding family of newly reducible second iterates.\\

We obtain the following immediate corollary of Theorem \ref{cubic2}. %, which addresses Question \ref{mainquest}.
\begin{corollary}
If $K$ is an infinite field satisfying the hypotheses of Theorem \ref{cubic2}, then $K \in \mathcal{N}_{3,2}^{\infty}$. In particular, if $K$ is a number field with ring of integers $\mathcal{O}_K$ and $2$ is not the cube of an element of $K$, then there are infinitely many monic $f \in \mathcal{O}_K[x]$ with newly reducible second iterate over $K$. 
\end{corollary}

%We end this section with the following question: 
%\begin{question}
%Does there exist monic $f \in \Q[x]$ such that $\deg f = 3$ and $f^2$ is newly reducible with three distinct factors, each of degree $3$?
%\end{question}
We end this section by briefly addressing Question \ref{cubicquest}, which asks whether there exists $f \in \Q[x]$ with $f$ irreducible but $f^2(x)$ a product of three irreducible cubics. The following proposition addresses this.
\begin{proposition}
Let $f(x) \in \Q[x]$ have the form in \eqref{cubicform}, and assume that $f^2$ is newly reducible with three distinct factors $p_1(x)$, $p_2(x)$, and $p_3(x)$. If $\beta$ is a root of $p_1(x + \gamma)$, and $f(\beta + \gamma) = \alpha$, then
\begin{equation}\label{cubic-root-equation}
\frac{-\beta \pm \sqrt{-3\beta^2 - 4a}}{2}
\end{equation}
are roots of $p_2(x + \gamma)$ and $p_3(x + \gamma)$ whose images under $f(x + \gamma)$ are $\alpha$.
\end{proposition}
\begin{proof}
Observe that each of the factors of $f^2(x)$ must have degree 3. By Lemma \ref{surjective}, the map   
\begin{equation*}
    \Phi_i : \{\text{roots of $p_i(x + \gamma)$ in $\overline{\Q}$}\} \to \{\text{roots of $f(x + \gamma)$ in $\overline{\Q}$}\}
\end{equation*} 
is surjective and $k$-to-one for each factor $p_i$ of $f^2$.  Since each factor has degree equal to the degree of $f$, $\Phi_i$ is one-to-one for each $i$, and there is one root of each $p_i$ whose image under $\Phi_i$ is $\alpha$.  We will show that the roots of $p_2(x + \gamma)$ and $p_3(x + \gamma)$ satisfying this condition are given by \eqref{cubic-root-equation}.  Let $\beta_2$ be such a root, and observe that $\beta^3 + a(\beta) + b + \gamma = \beta_2^3 + a(\beta_2) + b + \gamma$. This gives a quadratic equation for $\beta_2$ in terms of $\beta$, to which we apply the quadratic formula to obtain $\beta_2  = \frac{-\beta \pm \sqrt{-3\beta^2 - 4a}}{2}$.
\end{proof}
Notice that when $a = 0$, we have that the three preimages of a given root of $f(x)$ differ by a multiple of a cubic root of unity.  As a consequence, the constant terms of our three factors are equal.  We can equate coefficients using this fact to show that there is no monic cubic over $\Q[x]$ with $a = 0$ whose second iterate factors as three irreducible cubics.

\section{Quartics with newly reducible second iterate} \label{quartics}

In this section we will find infinitely many rational quartic polynomials with newly reducible second iterate. We use the following standard fact from field theory to tell when $f(x)$ is irreducible, noting that $K^n$ refers to the set $\{k^n : k \in K\}$.

\begin{proposition}{\cite[Theorem 8.1.6]{karpilovsky}}\label{xd-c-fact}
Let $K$ be a field and $f(x)=x^d-c \in K[x]$ for $d \geq 1$. Then $f(x)$ is irreducible over $K$ if and only if $c\not\in K^p$ for all primes $p\mid d$ and $c\not\in -4K^4$ whenever $4\mid d$.
\end{proposition}

%In this section we consider $d=4$, where the only prime dividing $d$ is 2. Therefore, as long as $c\not\in -4\Q^4$, $f(x)=x^4-c$ is irreducible if and only if $c\not\in\Q^2$. 
Throughout this section, we only consider polynomials of the form $f(x)=(x-\gamma)^4+m+\gamma$, where $m,\gamma\in\Q$. It turns out that even in this subset of rational quartics, there are infinitely many with newly reducible second iterate. %Note that in the notation of Fact \ref{xd-c-fact}, $c=-m-\gamma$.

\begin{theorem}\label{deg-4-theorem}
Let $K$ be a field of characteristic not equal to $2$, and suppose that $f(x)=(x-\gamma)^4+m+\gamma\in K[x]$ is irreducible over $K$. Then $f^2(x+\gamma)$ factors as $p(x^2)p(-x^2)$ for some $p \in K[x]$
%\begin{align}
%\begin{split}
%p_1(x) &= x^8 + dx^6 + cx^4 + bx^2 + a, \\
%p_2(x) &= x^8 - dx^6 + cx^4 - bx^2 + a,
%\end{split}\label{deg-4-p1p2}
%\end{align}
if and only if there exist $r,s\in K$ with $r \neq s^2$ such that $m = \frac{2 s^4-r^2}{8 r-8 s^2}$ and 
$$
\gamma = \frac{-2 r^5 s^2+19 r^4 s^4-72 r^3 s^6+32 r^3+136 r^2 s^8-32 r^2 s^2-128 r s^{10}-64 r s^4+48 s^{12}+64 s^6}{256\left(r-s^2\right)^2}.
$$
\end{theorem}

\begin{proof}
First, suppose that $m,\gamma,r,s\in K$ are as in the theorem. Writing $r_1=r-s^2$, one computes that 
$f^2(x+\gamma) = p(x^2)p(-x^2)$,
where
\begin{align*}
 p(x) &=  x^4
-  s x^3
+ \frac{s^4- r_1^2}{4 r_1}x^2 
+ \frac{3 r_1^2 s- s^5}{8 r_1}x 
+ \frac{-4 r_1 s^6+10 r_1^2 s^4-12 r_1^3 s^2+r_1^4+s^8}{64 r_1^2}.
\end{align*}

Conversely, suppose that $f(x+\gamma) = p(x^2)p(-x^2)$ for $p(x) = x^4 + dx^3 + cx^2 + bx + a$. 
%Expanding $f^2$ gives
%\begin{align*}
%f^2(x+\gamma) &= \gamma +m^4+4 m^3 x^4+6 m^2 x^8+4 m x^{12}+m+x^{16}
%\end{align*}
%And expanding $p_1p_2$ gives
%\begin{align*}
%p_1(x)p_2(x) &= a^2+x^4 \left(2 a c-b^2\right)+x^8 \left(2 a-2 b d+c^2\right)+x^{12} \left(2 c-d^2\right)+x^{16}
%\end{align*}
Equating coefficients gives us the following system of equations.
\begin{align}
a^2 &= \gamma +m^4+m \label{deg-4-1}\\
2 a c-b^2 &= 4 m^3 \label{deg-4-2}\\
2 a-2 b d+c^2 &= 6 m^2 \label{deg-4-3}\\
2 c-d^2 &= 4 m\label{deg-4-4}
\end{align}
First we solve \eqref{deg-4-4} for $c$ and substitute this into the other equations. Then we do the same with \eqref{deg-4-3} and $b$, and then with \eqref{deg-4-2} and $a$. This leaves us with equation \eqref{deg-4-1}, which becomes
\begin{align}
\frac{1}{64} \left(3 d^4+8 d^2 m \pm 2 \sqrt{2} d^2 \sqrt{d^4+4 d^2 m+8 m^2}+8 m^2\right)^2 &= \gamma +m^4+m \label{deg-4-large-gamma}
\end{align}
Note that if $d=0$, this equation becomes $m^4=\gamma+m^4+m$, so $\gamma+m=0$. But then $f(x+\gamma)=x^4$ is reducible, which is a contradiction. So $d\neq 0$, and therefore $\sqrt{2(d^4+4 d^2 m+8 m^2)}\in K$. In other words, there is a $K$-rational solution to
\begin{align*}
S:\ y^2 &= 16m^2 + 8s^2m + 2s^4
\end{align*}
where $s=d$. Considering $s$ as a fixed parameter, this is a conic. Now we want to parametrize the $K$-rational points on $S$, so we use the homogeneous form of $S$,
\begin{align*}
\overline{S}:\ Y^2 &= 16M^2 + 8s^2MZ + 2s^4Z^2,
\end{align*}
and project through the point at infinity $[M:Y:Z]=[1:4:0]$. We parametrize our projection line by where it crosses the line $m=0$. Then each line is given by $y=4m+r$, where $r\in\Q$. Then we solve for the intersection point of this line and $S$, giving
\begin{align*}
m = m(r,s) = \frac{2s^4-r^2}{8r-8s^2} \qquad \text{and} \qquad 
y = y(r,s) = 4m+r = \frac{2s^4-r^2}{2r-2s^2}-r.
\end{align*}
Then we solve \eqref{deg-4-large-gamma} for $\gamma$. This gives
\begin{align*}
\gamma &= \frac{1}{64} \left(3 s^4 + 8 s^2 m + 8 m^2 \pm 2 s^2 y\right)^2 - m - m^4.
\end{align*}
Note that there is a plus or minus in front of the term containing $y$. The plus sign leads to
\begin{align*}
\gamma &= \frac{-2 r^5 s^2+19 r^4 s^4-72 r^3 s^6+32 r^3+136 r^2 s^8-32 r^2 s^2-128 r s^{10}-64 r s^4+48 s^{12}+64 s^6}{256\left(r-s^2\right)^2},
\end{align*}
as in the statement of the theorem. On the other hand, if we have the minus sign, let
$r^\prime = \frac{s^2 \left(2 s^2-r\right)}{s^2-r}$  and $s^\prime = s$. 
Plugging these in gives $m(r^\prime,s^\prime) = m(r,s)$ and $y(r^\prime,s^\prime)=-y(r,s)$, again giving $\gamma$ as in the statement of the theorem. 
\end{proof}

\begin{corollary}\label{deg-4-infinitely-many-corollary}
Let $K$ be a field of characteristic not in $\{2,3\}$, and such that $3$ is not the square of an element of $K$, and $-3$ is not the fourth power of an element of $K$. Then
$$
f(x) = \left(x + 192t^8 - 7t^2\right)^4 -192t^8
$$
has a newly reducible second iterate over $K$ for any $t \in K, t \neq 0$.
% there are infinitely many monic quartic polynomials $f(x)=(x-\gamma)^4+m+\gamma\in\Z[x]$ such that $f^2(x)$ is newly reducible.
\end{corollary}
\begin{proof}
Let $t \in K$ with $t \neq 0$, and in the notation of Theorem \ref{deg-4-theorem} take $r=48t^2$ and $s=4t$. Then from Theorem \ref{deg-4-theorem} we have that $f(x) = (x-\gamma)^4 + \gamma + m$ has $f^2(x)$ reducible over $K$, where $\gamma = -192t^8 + 7t^2$ and $m = -7t^2$. We now argue that $f(x)$ is irreducible over $K$. Observe that 
$$-m-\gamma = 192t^8 = (2^6 \cdot 3)t^8$$
If $-m-\gamma = k^2$ for $k \in K$, then $3$ is a square in $K$, while if $4(m+\gamma) = k^4$ for $k \in K$, then $-3$ is a fourth power in $K$. Proposition \ref{xd-c-fact} now shows that $f$ is irreducible over $K$. 
\end{proof}

\section{Higher-degree polynomials with newly reducible second iterate}\label{largerdegrees}

%Here we attempt to prove Conjecture \ref{largedegreeconjecture}.

\begin{theorem} \label{infinited}
Let $d\equiv 2 \pmod{4}$, let $p_1, \ldots, p_r$ be the distinct odd primes dividing $d$, let $K$ be a field, and put $K^n = \{k^n : k \in K\}$. Assume that $-1 \not\in K^2$ and $-4k^4 \not\in K^{p_i}$ for each $i = 1, \ldots, r$. Then %for each $k \in (K \setminus (K^{p_1} \cup \cdots \cup K^{p_r}))$ 
the polynomial $f(x) = \left(x-4k^4\right)^d+4k^4$ has newly reducible second iterate over $K$. 
\end{theorem}
\begin{proof}
We have $f^2(x) = \left(x-4k^4\right)^{d^2}+4k^4$.  Because $4 \mid d^2$, Proposition \ref{xd-c-fact} gives that $f^2(x)$ is reducible over $K$. 
%And this factors:
%\begin{align*}
%f^2\left(x+4k^4\right) &= x^{d^2}+4k^4 \\
%&= \left(x^{\frac{d^2}{2}} + 2kx^{\frac{d^2}{4}} + 2k^2\right)\left(x^{\frac{d^2}{2}} - 2kx^{\frac{d^2}{4}} + 2k^2\right)
%\end{align*}
%Because $4\mid d^2$. 
To show that $f(x)$ is irreducible over $K$, we note that $4 \nmid d$, and so by Proposition \ref{xd-c-fact} it is enough to show that $-4k^4 \not\in K^2$ and $-4k^4 \not\in K^{p_i}$ for each $i = 1, \ldots, r$. The former follows because $-1 \not\in K^2$ and the latter follows by hypothesis. 
\end{proof}

\begin{corollary} \label{genbigd}
Let $K$ be a field with a non-trivial discrete valuation. Assume that $K$ has characteristic different from $2$, and that $-1 \not\in K^2$. 
%Let $R$ be a Dedekind domain that is not a field, and let $K$ be the field of fractions of $R$. Assume that $R$ has characteristic different from $2$, and that $-1 \not\in K^2$. 
Then $K \in \mathcal{N}_{d,2}^\infty$ for all $d \equiv 2 \pmod{4}$. 
\end{corollary}

\begin{proof}
%Because $R$ is not a field, it has a non-zero maximal ideal $\mathfrak{p}$. 
Let $v : K^* \to \mathbb{Z}$ be a (normalized) non-trivial discrete valuation on $K$. This map is surjective, and taking $\pi$ with $v(\pi) = 1$, we also have that $v^{-1}(m)$ maps bijectively to $v^{-1}(n)$ via the map $x \mapsto \pi^{m-n}x$.  Observe that $2 \in K^*$ because $K$ has characteristic $\neq 2$, and that if $k \in K^n$ for any $n \geq 2$, then $v(k) \in n\Z$. Fix $d \equiv 2 \pmod{4}$, and let $p_1, \ldots, p_r$ be the odd prime divisors of $d$. By the Chinese Remainder Theorem, the system of congruences
$$
x \equiv  \frac{1-2v(2)}{4}  \pmod{p_i}
$$
has an infinite solution set $S \subset \Z$. If $v(k) \in S$, then $2v(2) + 4v(k) \equiv 1 \pmod{p_i}$ for all $i$, and hence $v(-4k^4) \not\in p_i\Z$, and thus $-4k^4 \not\in K^{p_i}$ for all $i$. Hence by Theorem \ref{infinited} we have that $f(x) = \left(x-4k^4\right)^d+4k^4$ has newly reducible second iterate over $K$.

%Let $(q_i)_{i \geq 0}$ be an enumeration of the infinitely many prime numbers that do not divide $d$ or $v(2)$ (the latter is well-defined since $R$ has characteristic different from $2$.) If $v(k) = q_i$ for some $i$, then $v(-4k^4) = 8q_iv(2)$, we cannot have $k \in K^{p}$ for any prime $p \mid d$. Hence if $k \in V := \bigcup_{i=1}^\infty v^{-1}(q_i)$, then by  

It remains to argue that $v^{-1}(S)$, and hence $S$, is infinite. We have already established that $v^{-1}(m)$ and $v^{-1}(n)$ are equinumerous for any $m,n \in \Z$. The only way an infinite union of equinumerous sets can be finite is if all are empty, but this contradicts the surjectivity of $v$. Thus $v^{-1}(S)$ is infinite.  
\end{proof}

\bibliographystyle{plain}

\end{document}